\documentclass[12pt]{amsart}
\usepackage{graphicx}
\newtheorem{theorem}{Theorem}[section]
\newtheorem{proposition}[theorem]{Proposition}
\newtheorem{lemma}[theorem]{Lemma}

\theoremstyle{definition}
\newtheorem{definition}[theorem]{Definition}

\newcommand{\bdd}{\mbox{$\partial$}}

\theoremstyle{remark}
\newtheorem{remark}[theorem]{Remark}

\numberwithin{equation}{section}

\begin{document}

\title[$1$-genus $1$-bridge knots II]
{Genus two Heegaard splittings
of exteriors of \\$1$-genus $1$-bridge knots II}

\author{Hiroshi Goda and Chuichiro Hayashi}

\date{\today}

\thanks{The first and second authors are partially supported
by Grant-in-Aid for Scientific Research, (No. 21540071 and 18540100),
Ministry of Education, Science, Sports and Culture.}

\begin{abstract}
A knot $K$ is called a $1$-genus $1$-bridge knot  in a $3$-manifold $M$ 
if $(M,K)$ has a Heegaard splitting $(V_1,t_1)\cup (V_2,t_2)$  
where $V_i$ is a solid torus and $t_i$ is a  boundary parallel arc 
properly embedded in $V_i$. 
If the exterior of a knot has a genus 2 Heegaard splitting, 
we say that the knot has an unknotting tunnel. 
Naturally the exterior of a $1$-genus $1$-bridge knot $K$ allows 
a genus 2 Heegaard splitting, i.e., 
$K$ has an unknotting tunnel. 
But, in general, there are unknotting tunnels which are not 
derived form this procedure. 
Some of them may be levelled with the torus $\partial V_1=\partial V_2$, 
whose case was studied in our previous paper \cite{GH2}.  
In this paper, we consider the remaining case.
\end{abstract}

\maketitle

\section{Introduction}
This paper is a sequel to our previous paper \cite{GH2}.
We will use the same notations and terminology. 
Let $K$ be a $1$-genus $1$-bridge knot in the $3$-sphere $S^3$, and 
suppose that there are  a $(1,1)$-splitting
$(S^3, K) = (V_1, t_1) \cup_{H_1} (V_2, t_2)$
and a $(2,0)$-splitting
$(S^3, K) = (W_1, K) \cup_{H_2} (W_2, \emptyset)$.
We denote by $\ell$ the number of $K$-essential loops of 
the intersection $H_1$ and $H_2$. 
The number $\ell$ is said to be {\it minimum\/}
if there is no isotopies of $H_1$ and $H_2$ in $(M,K)$ 
so that they intersect each other in non-empty collection 
of smaller number of loops which are $K$-essential 
both in $H_1$ and in $H_2$. 
In \cite{GH2}, 
we proved the following theorem.

\begin{theorem}[\cite{GH2}]\label{thm:main}
Suppose $\ell$ is minimum and $\ell\neq 2$. 
Then, at least one of the following conditions holds.
\begin{enumerate}
\renewcommand{\labelenumi}{(\theenumi)}
\item
The $(2,0)$-splitting $H_2$ is meridionally stabilized.
\item
There is an arc $\gamma$
which forms a spine of $(W_1, K)$
and is isotopic into the torus $H_1$.
Moreover, we can take $\gamma$
so that
there is a canceling disk $C_i$ of the arc $t_i$ in $(V_i, t_i)$
with $\bdd C_i \cap \gamma = \partial \gamma = \partial t_i$
for $i=1$ or $2$.
\item
The $(1,1)$-splitting $H_1$ admits a satellite diagram
of a longitudinal slope.
\end{enumerate}
\end{theorem}

\begin{figure}[htbp]
\begin{center}
\includegraphics[width=.7\textwidth]{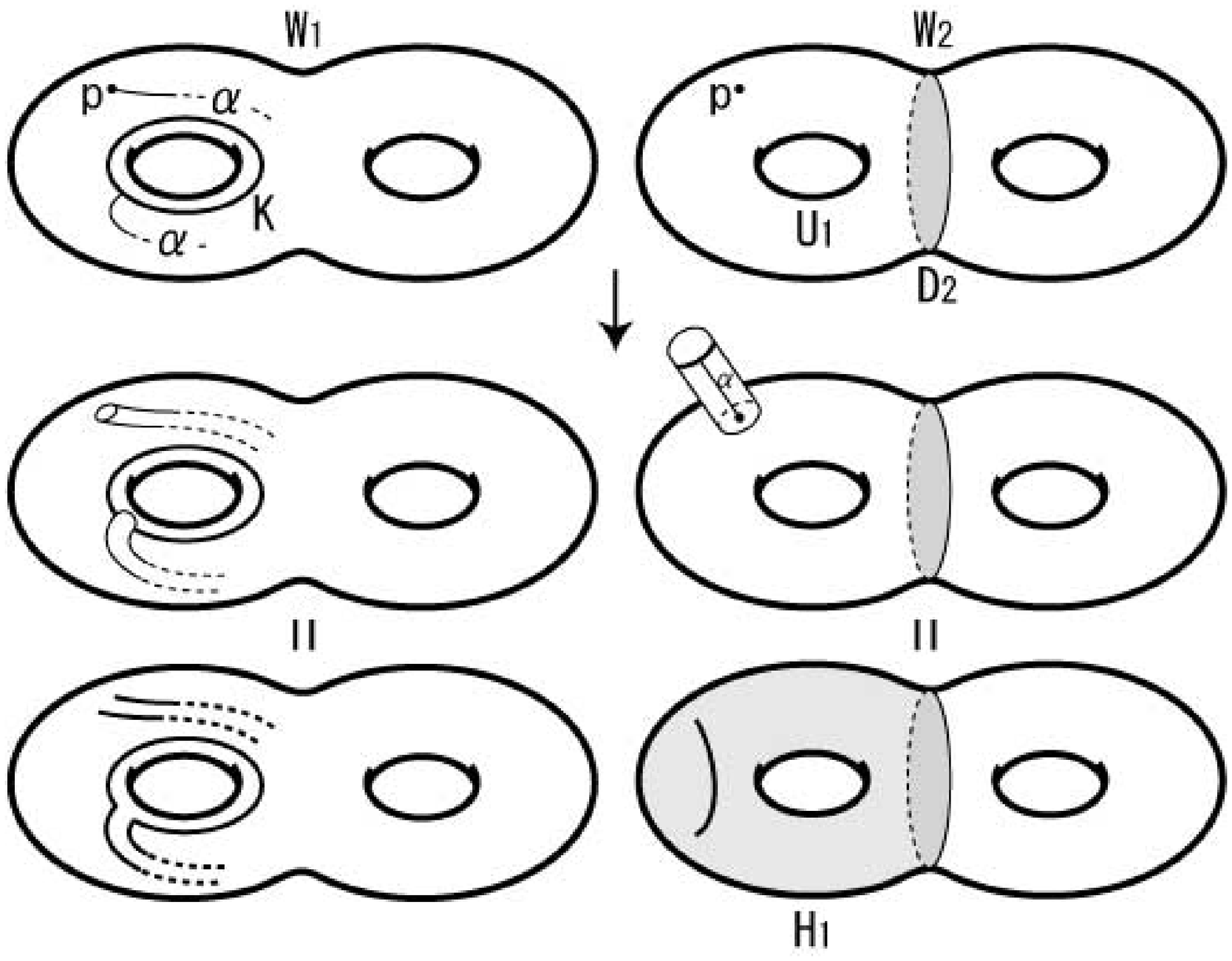}
\end{center}
\caption{}
\label{fig:Case3}
\end{figure}

In this paper, we consider the case of $\ell=2$
in a slightly general setting.
Concretely, we will prove the following theorem. 


\begin{theorem}\label{thm:l=2}
Let $M$ be the $3$-sphere or a lens space $($other than $S^2 \times S^1)$,
and $K$ a knot in $M$.
Let $(V_1, t_1) \cup_{H_1} (V_2, t_2)$ and  $(W_1, K) \cup_{H_2} (W_2, \emptyset)$
be a $(1,1)$-splitting and $(2,0)$-splitting of $(M,K)$.
Suppose that the splitting surfaces $H_1$ and $H_2$
intersect each other in $2$ loops
which are $K$-essential both in $H_1$ and in $H_2$.
Further, we assume that $M$ has a $2$-fold branched cover with branch set $K$.
Then at least one of 
the six conditions $(a) \sim (f)$ below holds.
\begin{enumerate}
\item[(a)]
 We can isotope $H_1$ and $H_2$ in $(M, K)$
so that they intersect in one loop
which is $K$-essential both in $H_1$ and in $H_2$.
\item[(b)]
 The $(2,0)$-splitting $H_2$ is weakly $K$-reducible.
\item[(c)]
 The knot $K$ is a torus knot.
\item[(d)]
 The knot $K$ is a satellite knot.
\item[(e)]
 The $(1,1)$-splitting $H_1$ admits a satellite diagram of a longitudinal slope.
\item[(f)]
 There is an essential separating disk $D_2$ in $W_2$,
and an arc $\alpha$ in $W_1$
such that $\alpha \cap K$ is one of the endpoints $\bdd \alpha$,
and 
$\alpha \cap \bdd W_1$ is the other endpoint, say $p$, of $\alpha$
and that $D_2$ cuts off a solid tours $U_1$ from $W_2$
with $p \in \bdd U_1$
and with the torus $\partial N(U_1 \cup \alpha)$
isotopic to the $(1,1)$-splitting torus $H_1$ in $(M, K)$.
See Figure \ref{fig:Case3}.
\end{enumerate}
\end{theorem}

This theorem provides the case (3-1) of Theorem 1.3 in \cite{GH2} 
so that the proof completes. 

The authors thank 
the anonymous referee for his-or-her 
helpful comments.

\section{Preliminaries}

In this paper, we will use the same notations and terminology 
as in the previous paper \cite{GH2}.
In the proof of Theorem \ref{thm:l=2}, 
we use the next proposition.
The condition
that $M$ has a $2$-fold branched cover with branch set $K$ 
is necessary 
only when we apply it.

\begin{proposition}[Proposition 3.4 in \cite{K2}]\label{thm:koba}
 Let $M$ be a closed orientable 3-manifold, and $L$ a link in $M$.
 Assume that $M$ has a 2-fold branched cover
with branch set $L$.
 Let $H_i$ be 
$(g_i, n_i)$-splitting of $(M,L)$
for $i=1$ and $i=2$,
and $W$ a genus $g_2$ handlebody bounded by $H_2$ in $M$.
 Suppose that $H_1$ is contained in the interior of $W$,
and that there is
an $L$-compressing or meridionally compressing disk $D$ of $H_2$
in $(W, L \cap W)$ with $D \cap H_1 = \emptyset$.
 Then either
$(i)$ $M = S^3$ and $L = \emptyset$ or $L$ is the trivial knot,
or $(ii)$ the splitting $H_2$ is weakly $L$-reducible.
\end{proposition}

Let $M$ be the $3$-sphere or a lens space $($other than $S^2 \times S^1)$,
and $K$ a knot in $M$.
Let $(V_1, t_1) \cup_{H_1} (V_2, t_2)$ and  $(W_1, K) \cup_{H_2} (W_2, \emptyset)$
be a $(1,1)$-splitting and $(2,0)$-splitting of $(M,K)$. 
According to the assumption of Theorem \ref{thm:l=2}, 
we suppose that 
$H_1 \cap H_2$ consists of two loops,
say $l_1$ and $l_2$, which are $K$-essential
both in $H_1$ and in $H_2$.

Since the $(2,0)$-splitting surface $H_2$ separates $M$,
there are two patterns of intersection loops of $H_1 \cap H_2$
in $H_1$:
(1) $l_1$ and $l_2$ together divide $H_1$
into a disk $Q$, an annulus $A_1$ and a torus with one hole $H'_1$,
or
(2) $l_1$ and $l_2$ together divide $H_1$
into two annuli $A_{11}$ and $A_{12}$.
 Since the $(1,1)$-splitting torus $H_1$ also separates $M$,
there are two patterns of intersection loops of $H_1 \cap H_2$
in $H_2$:
(A) $l_1$ and $l_2$ are parallel separating essential loops,
or
(B) they are parallel non-separating essential loops.
 See Figure \ref{fig:l=2}. 
In the following sections, we study each case. 

\begin{figure}[htbp]
\centering
\includegraphics[width=.7\textwidth]{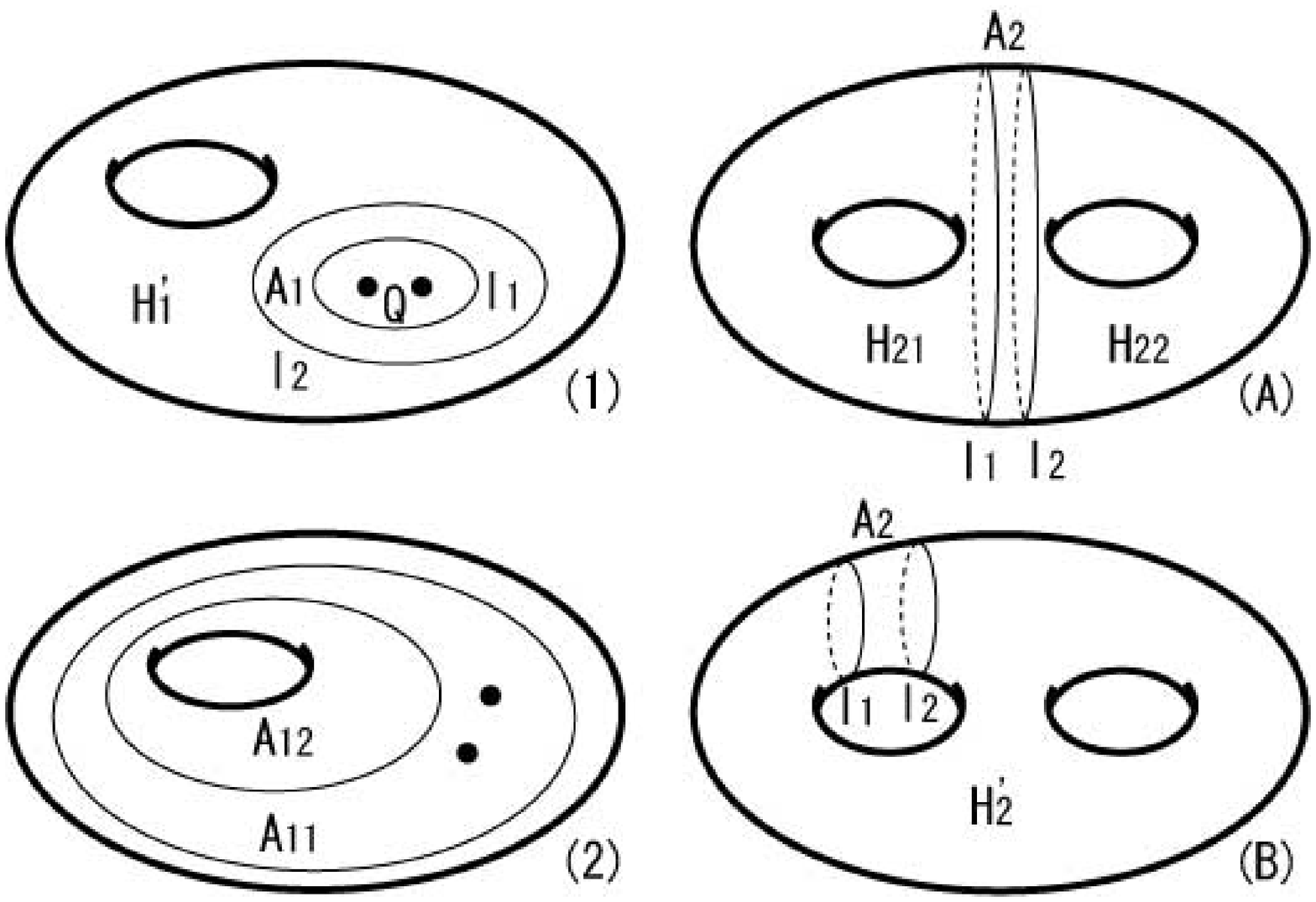}
\label{fig:l=2}
\caption{}
\end{figure}


\section{Case (1)(A)}\label{sect:1A}

 In this section, we consider the case
where $H_1 \cap H_2$ is of the pattern (1) in $H_1$,
and of the pattern (A) in $H_2$ in  Figure \ref{fig:l=2}.
 In Case (1),
we assume, without loss of generality,
that $l_1 = \bdd Q$.
 The disk $Q$ contains the two intersection points $K \cap H_1$
because the loop $l_1$ is $K$-essential in $H_1$.
 Then the torus with one hole $H'_1$ and the annulus $A_1$
are disjoint from $K$.
 In particular, the loops $l_1$ and $l_2$ are parallel in $H_1 -K$.
 These loops are $K$-essential but inessential in $H_1$.
 Since the handlebody $W_1$ contains $K$ as a core,
and since $Q$ intersects $K$,
$Q$ and $H'_1$ is contained in $W_1$,
and $A_1$ in $W_2$.
 In Case (A),
for $i=1$ and $2$
the loop $l_i$ bounds a torus with one hole, say $H_{2i}$,
such that $H_{21} \cap H_{22} = \emptyset$.
 The complementary region is an annulus, say $A_2$,
between $l_1$ and $l_2$.
 We can assume, without loss of generality,
that $A_2$ is contained in the solid torus $V_1$,
and $H_{21} \cup H_{22}$ in $V_2$.

 By Lemma 3.1
in \cite{GH2}
(Lemma 3.10 in \cite{GHY}),
$Q \cup H'_1$ is $K$-compressible or $K$-$\bdd$-compressible
in $(W_1, K)$.
 $A_1$ is compressible or $\bdd$-compressible
in the handlebody $W_2$.

\begin{lemma}\label{lem:1a-A1compressible}
 Suppose that $A_1$ is compressible in $W_2$,
and that $W_1$ contains either 
$(0)$ a $t_1$-compressing disk of $A_2$ in $(V_1, t_1)$, 
$(1)$ a $t_2$-compressing disk of $H_{21} \cup H_{22}$ in $(V_2, t_2)$,
$(2)$ a $t_2$-$\bdd$-compressing disk of 
$H_{21} \cup H_{22}$ incident to $H_{21}$ in $(V_2, t_2)$,
$(3)$ a $K$-compressing disk of $Q \cup H'_1$ in $(W_1, K)$
or 
$(4)$ a $K$-$\bdd$-compressing disk of $Q \cup H'_1$ 
incident to $Q$ in $(W_1, K)$.
 Then the $(2,0)$-splitting $H_2$ of $(M,K)$ is weakly $K$-reducible,
\end{lemma}

\begin{proof}
 Compressing $A_1$,
we obtain a disk $D_2\,(\subset W_2)$ bounded by $l_1$.

 In Case (2), 
the $t_2$-$\bdd$-compressing disk is also 
a $K$-$\bdd$-compressing disk of $Q$
by the unusual definition of $\bdd$-compressing disk.
 (See Definition 2.3. in \cite{GH2}.)
 Hence, in Cases (2), (3) and (4),
we obtain a $K$-compressing or meridionally compressing disk
of $H_{21} \cup H_{22}$ in $(W_1, K)$
by performing a compressing or $\bdd$-compressing operation
on a copy of $Q$ or $H'_1$.
 Hence, in all cases, there is 
a $K$-compressing or meridionally compressing disk $D_1$ 
of $H_{21}$, $H_{22}$ or $A_2$ in $(W_1, K)$.
 Then the disks $D_1$ and $D_2$ together show
that $H_2$ is weakly $K$-reducible.
\end{proof}

\begin{lemma}\label{lem:1a-A1incompressible}
 Suppose that $A_1$ is incompressible in $W_2$
and that $W_1$ contains either 
$(1)$ a $t_2$-compressing disk of $H_{21} \cup H_{22}$,
$(2)$ a $t_2$-$\bdd$-compressing disk of 
$H_{21} \cup H_{22}$ incident to $H_{21}$,
$(3)'$ a $K$-compressing disk of $Q \cup H'_1$ in $V_2$
or $(4)$ a $K$-$\bdd$-compressing disk of $Q \cup H'_1$ incident to $Q$.
 Then $H_2$ is weakly $K$-reducible,
\end{lemma}

\begin{proof}
 In Case (4),
the $K$-$\bdd$-compressing disk is incident to $H_{21}$ rather than $A_2$
by the definition of $\bdd$-compressing disk.
 Thus, in all cases,
we can assume that $V_2 \cap W_1$ contains
a $K$-compressing or meridionally compressing disk $D$ of $H_{21} \cup H_{22}$
in $(W_1, K)$
by a similar argument as in the second paragraph 
in the proof of Lemma \ref{lem:1a-A1compressible}.

 Since the annulus $A_1$ is incompressible in the handlebody $W_2$,
it is $\bdd$-compressible in $W_2$.
 Hence $A_1$ is parallel to $A_2$ in $W_2$.
 We can isotope $H_1$ along the parallelism
so that $H_1$ is contained in int\,$W_1$ and is disjoint from $D$.
 Then Proposition \ref{thm:koba} shows
that $H_2$ is weakly $K$-reducible.
\end{proof}



\begin{definition}
We call a $(2,0)$-splitting
$(M, K) = (W_1, K) \cup_{H_2} (W_2, \emptyset)$
{\it semi-stabilized}
if there is a $K$-compressing disk $D_i$ of $H_2$ in $(W_i, K \cap W_i)$
for $i=1$ and $2$
such that $\bdd D_1$ and $\bdd D_2$ intersect each other
transversely in precisely two points.
\end{definition}

\begin{proposition}[Theorem 7.2 in \cite{GHY}]\label{prop:semistabilized}
If $(M, K)$ admits a semi-stabilized strongly $K$-irreducible
$(2,0)$-splitting, then one of the following occurs:
\begin{enumerate}
\renewcommand{\labelenumi}{(\theenumi)}
\item
the knot $K$ is a torus knot in $M$;
\item
the knot $K$ is a satellite knot;
\item
the $3$-manifold $M$ admits a Seifert fibered structure
over the $2$-sphere with three exceptional fibers,
and $K$ is an exceptional fiber; or
\item
the exterior of the knot $K$ contains a non-separating torus.
\end{enumerate}
\end{proposition}
When $M$ is the $3$-sphere or a lens space except for $S^2 \times S^1$,
the conclusions $(3)$ and $(4)$ do not occur.

\begin{lemma}\label{lem:1a-H'_1H_22}
 Suppose 
that $V_2 \cap W_1$ contains
a $t_2$-$\bdd$-compressing disk of 
$H_{21} \cup H_{22}$ incident to $H_{22}$
or a $K$-$\bdd$-compressing disk of $Q \cup H'_1$ incident to $H'_1$.
 Then one of the following conditions holds.
\begin{enumerate}
\renewcommand{\labelenumi}{(\theenumi)}
\item
 $H_2$ is weakly $K$-reducible.
\item
 We can isotope $H_1$ in $(M,K)$
so that $H_1 \cap H_2$ is a single loop
which is $K$-essential both in $H_1$ and in $H_2$.
\item
 $H_2$ is semi-stabilized,
and $K$ is a torus knot or a satellite knot.
\end{enumerate}
 Moreover,
if either $A_1$ is compressible in $W_2$
or $Q \cup H'_1$ is $K$-incompressible, 
then $(1)$ or $(2)$ holds.
\end{lemma}

\begin{figure}[htbp]
\begin{center}
\includegraphics[width=.7\textwidth]{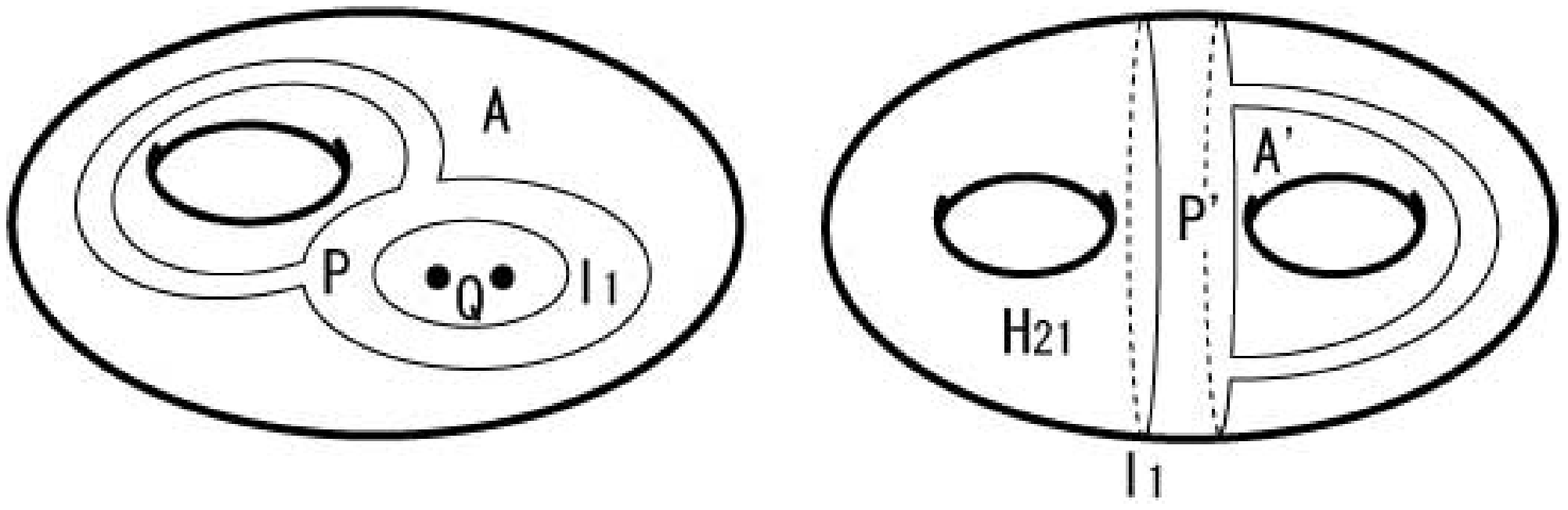}
\end{center}
\caption{}
\label{fig:10-2}
\end{figure}

\begin{proof}
 We isotope $H_1$ along the $\bdd$-compressing disk $D$
slightly beyond the arc $\bdd D \cap H_{22}$.
 Then, 
by the definition of $\bdd$-compressing disk,
$H'_1$ is deformed into an annulus $A$,
each of the boundary loops $\bdd A$ is non-separating in $H_{22}$,
and these loops cobound an annulus, say $A'$, in $H_{22}$.
 The annulus $A_1$ is deformed into a disk with two holes, say $P$,
in the handlebody $W_2$.
 Set $P' = \text{cl}\,(H_2 -(H_{21} \cup A'))$.
 See Figure \ref{fig:10-2}.
 Note that $P$ is parallel to $P'$ in $W_2$
if $A_1$ is $\bdd$-compressible in $W_2$ before the isotopy.
 By Lemma 3.1 in \cite{GH2}, 
$Q \cup A$ is $K$-compressible or $K$-$\bdd$-compressible
in $(W_1, K)$.

{\it Case $(a)$.}
 We first assume that
$Q \cup A$ is $K$-incompressible and 
has a $K$-$\bdd$-compressing disk, say $R$,
in $(W_1, K)$.
 (This holds if $Q \cup H'_1$ is $K$-incompressible before the isotopy.)
 If the arc $\bdd R \cap H_2$ is contained in $H_{21}$,
then $R$ is there also before the isotopy,
and $H_2$ is weakly $K$-reducible
by Lemmas \ref{lem:1a-A1compressible} and \ref{lem:1a-A1incompressible}.
 If the arc $\bdd R \cap H_2$ is contained in the annulus $A'$,
then $A$ is parallel to $A'$ in $W_1 - K$.
 We can isotope $H_1$ along the parallelism
so that $H_1$ and $H_2$ intersect each other only in the loop $l_1$
which is $K$-essential both in $H_1$ and in $H_2$.
 Hence we may assume
that the arc $\bdd R \cap H_2$ is contained
in the disk with two holes $P'$.
 If the disk $R$ is incident to the annulus $A$,
then performing a $K$-$\bdd$-compressing operation on $A$ along $R$,
we obtain a disk, say $R'$, disjoint from $K$.
 Since the arc $\bdd R \cap H_2$ is an essential arc in $P'$
and since it connects the two loops $\bdd A'$,
the boundary loop $\bdd R'$ is parallel to $l_1 = \bdd Q$.
 This implies that the disk $Q$ is $K$-compressible in $(W_1, K)$.
 This contradicts our assumption.
 Hence $R$ is incident to $Q$.
 $\bdd$-compressing $Q$ along $R$,
we obtain a disk, say $R_1$,
which intersects $K$ transversely in a single point.
 We can isotope $R_1$ in $(W_1, K)$
so that it is bounded by a component of $\bdd A'$.
 If $A_1$ is compressible before the isotopy,
then a compressing operation on $A_1$ yields a disk bounded by $l_1$ in $W_2$.
 This disk together with $R_1$ shows
that $H_2$ is weakly $K$-reducible.

 Hence we can assume that $A_1$ is incompressible
and $\bdd$-compressible in $W_2$. 
 In this case, recall that $P$ is parallel to $P'$ in $W_2$.
 By Lemma 2.4 in \cite{GH2} (Lemma 2.8 in \cite{Hy3}),  
$H_{21} \cup A'$ is $t_2$-compressible or $t_2$-$\bdd$-compressible
in $(V_2, t_2)$.
 First we consider the former case.
 If $H_{21} \cup A'$ has a $t_2$-compressing disk in $W_1$,
then compressing a copy of $H_{21} \cup A'$,
we obtain a $K$-compressing disk of $Q \cup A$,
which contradicts our assumption.
 Hence $H_{21} \cup A'$ has a $t_2$-compressing disk in $W_2$.
 This disk together with $R_1$ shows
that $H_2$ is weakly $K$-reducible.

 We consider the latter case,
where $H_{21} \cup A'$ has a $t_2$-$\bdd$-compressing disk, say $C$.
 If $C$ is contained in $W_1$,
then it is also a $K$-$\bdd$-compressing disk of $Q \cup A$.
 When $C$ is incident to $A'$,
the annulus $A$ is parallel to $A'$ in $(W_1, K)$,
and we obtain the conclusion (2) of this Lemma.
 When $C$ is incident to $H_{21}$, 
it is a $K$-$\bdd$-compressing disk of $H_{21}$ before the isotopy,
and Lemmas \ref{lem:1a-A1compressible} and \ref{lem:1a-A1incompressible} imply
that $H_2$ is weakly $K$-reducible.
 Hence we may assume that $C$ is contained in $W_2$.
 If $C$ is incident to $H_{21}$,
then this disk is extended to an essential disk
with boundary loop in $H_{21} \cup P'$
because $P$ is parallel to $P'$ in $W_2$.
 This disk together with $R_1$
shows that
$H_2$ is weakly $K$-reducible.
 If $C$ is incident to the annulus $A'$,
then this disk is extended to an essential disk in $W_2$
such that its boundary loop intersects $\bdd R_1$
transversely in a single point
because $P$ is parallel to $P'$ in $W_2$.
 Then this disk together with $R_1$
shows that $H_2$ is meridionally stabilized,
and hence is weakly $K$-reducible
by Proposition 4.7 in \cite{GH2}
(Proposition 2.14 in \cite{GHY}).

{\it Case $(b)$.}
 We consider the case 
where $Q \cup A$ has a $K$-compressing disk $E$ in $(W_1, K)$.
 Note that $E$ gives a $K$-compressing disk of $Q \cup H'_1$ before the isotopy.
 When $A_1$ is compressible in $W_2$,
Lemma \ref{lem:1a-A1compressible}
shows that $H_2$ is weakly $K$-reducible.
 Hence we can assume that $A_1$ is incompressible 
and $\bdd$-compressible in $W_2$.
 Recall that $P$ is parallel to $P'$ in $W_2$ after the isotopy.
 By Lemma \ref{lem:1a-A1incompressible},
we can assume that $E$ is contained in $W_1 \cap V_1$.
 Then a $K$-compressing operation on $Q \cup H'_1$
yields a $K$-compressing disk $E'$ of the annulus $A_2$ in $(W_1, K)$.
 After an adequate isotopy,
we may assume that $\bdd E'= l_1$ 
and $E'$ is disjoint from (int\,$Q) \cup H'_1$.
 After the isotopy along $D$,
$H_{21} \cup A'$ is $t_2$-compressible 
or $t_2$-$\bdd$-compressible in $(V_2, t_2)$
by Lemma 2.4 in \cite{GH2}.
 We first consider the case 
where $H_{21}\cup A'$ is $t_2$-compressible.
 If the compressing disk is in $W_2$,
then it shows that $H_2$ is weakly $K$-reducible
together with $E'$.
 If it is in $W_1$,
then Lemma \ref{lem:1a-A1incompressible} shows
that $H_2$ is weakly $K$-reducible.

 Hence we can assume 
that $H_{21} \cup A'$ is $t_2$-incompressible in $(V_2, t_2)$, 
and has a $t_2$-$\bdd$-compressing disk, say $Z$,
in $(V_2, t_2)$.
 First suppose
that $Z$ is incident to $A'$.
 If $Z$ is contained in $W_1$,
then $A$ is parallel to $A'$ in $W_1 - K$,
and the conclusion (2) holds.
 If $Z$ is contained in $W_2$,
then this disk gives an essential disk in $W_2$
such that its boundary loop is disjoint from $\bdd E'$,
since $P$ is parallel to $P'$ in $W_2$.
 This shows that $H_2$ is weakly $K$-reducible.

 Thus we may assume
that $Z$ is incident to $H_{21}$.
 If $Z$ is contained in $W_1$,
then $H_2$ is weakly $K$-reducible
by Lemma \ref{lem:1a-A1incompressible}.
 Hence we may assume that $Z$ is contained in $W_2$.
 The boundary loop $\bdd Z$ intersects $P$
in an essential arc with both endpoints in $l_1$.
 Since $P$ and $P'$ are parallel in $W_2$,
$Z$ gives an essential disk in $W_2$
such that its boundary loop intersects $\bdd E'$
transversely in two points.
 Thus $H_2$ is semi-stabilized,
and the conclusion (3) holds by Proposition \ref{prop:semistabilized}.
\end{proof}

\begin{lemma}\label{lem:1a-A1compressible2}
 Suppose that $A_1$ is compressible in $W_2$.
 Then either $H_2$ is weakly $K$-reducible,
or
we can isotope $H_1$ in $(M,K)$
so that $H_1 \cap H_2$ is a single loop
which is $K$-essential both in $H_1$ and in $H_2$.
\end{lemma}

\begin{proof}
 $Q \cup H'_1$ is $K$-compressible or $K$-$\bdd$-compressible in $(W_1, K)$ 
by Lemma 3.1 in \cite{GH2}.
 In the former case, we have the conclusion 
by Lemma \ref{lem:1a-A1compressible}.
 In the latter case, 
the $K$-$\bdd$-compressing disk is not incident to $A_2$ 
by the definition of $\bdd$-compressing disk. 
 Hence the conclusion holds 
by Lemmas \ref{lem:1a-A1compressible} and \ref{lem:1a-H'_1H_22}.
\end{proof}

\begin{lemma}\label{lem:1a-QHcompressible}
 Suppose that $A_1$ is incompressible in $W_2$. 
and that $Q \cup H'_1$ is $K$-compressible in $(W_1, K)$.
 Then one of the following conditions holds.
\begin{enumerate}
\renewcommand{\labelenumi}{(\theenumi)}
\item
 $H_2$ is weakly $K$-reducible.
\item
 We can isotope $H_1$ in $(M,K)$
so that $H_1 \cap H_2$ is a single loop
which is $K$-essential both in $H_1$ and in $H_2$.
\item
 $H_2$ is semi-stabilized,
and $K$ is a torus knot or a satellite knot.
\end{enumerate}
\end{lemma}

\begin{proof}
 Compressing $Q \cup H'_1$
we obtain a disk, say $D$, disjoint from $K$.
 An adequate isotopy moves $D$ so that $\bdd D = l_1$.

 $H_{21} \cup H_{22}$
is $t_2$-compressible or $t_2$-$\bdd$-compressible in $(V_2, t_2)$
by Lemma 2.4 in \cite{GH2}.
 In the former case,
let $R$ be a $t_2$-compressing disk of $H_{21} \cup H_{22}$.
 If $R$ is contained in $W_1$, 
then $H_2$ is weakly $K$-reducible by Lemma \ref{lem:1a-A1incompressible}.
 If $R$ is contained in $W_2$,
the disks $D$ and $R$ together show
that $H_2$ is weakly $K$-reducible.

 Hence we may assume
that $H_{21} \cup H_{22}$ 
has a $t_2$-$\bdd$-compressing disk $R'$ in $(V_2, t_2)$.
 Since $H_1 \cap W_2 =A_1$ is an annulus disjoint from $K$,
the disk $R'$ is contained in $W_1$ rather than $W_2$
because of the definition of $t_2$-$\bdd$-compressing disk.
 Thus $R'$ is contained in $V_2 \cap W_1$,
and we obtain the conclusion
by Lemmas \ref{lem:1a-A1incompressible} and \ref{lem:1a-H'_1H_22}.
\end{proof}

\begin{lemma}\label{lem:1a-QHincompressible}
 Suppose that $A_1$ is incompressible in $W_2$
and that $Q \cup H'_1$ is $K$-incompressible in $(W_1, K)$.
 Then either $H_2$ is weakly $K$-reducible,
or
we can isotope $H_1$ in $(M,K)$
so that $H_1 \cap H_2$ is a single loop
which is $K$-essential both in $H_1$ and in $H_2$.
\end{lemma}

\begin{proof}
 Because $Q \cup H'_1$ is $K$-incompressible in $(W_1, K)$,
it has a $K$-$\bdd$-compressing disk, say $D$.
 By the definition of $K$-$\bdd$-compressing disk,
the arc $\bdd D \cap H_2$ is contained in $H_{21}$ or $H_{22}$
rather than in $A_2$.
 Then we obtain the conclusion 
by Lemmas \ref{lem:1a-A1incompressible} and \ref{lem:1a-H'_1H_22}.
\end{proof}

 Thus, in Case (1)(A), we have
Conclusion (a), (b), (c) or (d) of Theorem \ref{thm:l=2}
by Lemmas \ref{lem:1a-A1compressible2}, 
\ref{lem:1a-QHcompressible}
and \ref{lem:1a-QHincompressible}.


\section{Case (1)(B)}\label{sect:1B}

 In this section, we consider the case
where the intersection loops $H_1 \cap H_2$ are
of the pattern (1) in $H_1$
and of the pattern (B) in $H_2$ in Figure \ref{fig:l=2}.
 In Case (B),
for $i=1$ and $2$
each of the loops $l_1$ and $l_2$ of $H_1 \cap H_2$
is non-separating in $H_2$,
and they cobound an annulus, say $A_2$.
 The complementary region $H'_2 = \text{cl}\,(H_2 - A_2)$
is a torus with two holes.
 Let $Q, A_1, H'_1$ be as in the previous section.
 In particular, $l_1 = \bdd Q$ and $l_2 = \bdd H'_1$.
 We may assume, without loss of generality,
that $A_2$ is contained in $V_1$,
and $H'_2$ in $V_2$.
 See Figure \ref{fig:l=2}. 

 By Lemma 3.1 in \cite{GH2},
$Q \cup H'_1$ is $K$-compressible or $K$-$\bdd$-compressible
in $(W_1, K)$.
 $A_1$ is compressible or $\bdd$-compressible in the handlebody $W_2$.
 When $A_1$ is $\bdd$-compressible,
either $A_1$ is parallel to $A_2$ in $W_2$,
or $A_1$ has a $\bdd$-compressing disk
that is also a $t_2$-$\bdd$-compressing disk of $H'_2$
in $(V_2, t_2)$.

\begin{lemma}\label{lem:1b-incompressible2}
 Suppose that $A_1$ is incompressible
and not parallel to $A_2$ in $W_2$.
 Then $H_2$ is weakly $K$-reducible.
\end{lemma}

\begin{proof}
 Since the annulus $A_1$ is incompressible
and not parallel to $A_2$ in $W_2$,
it has a $\bdd$-compressing disk $D$
such that the arc $\bdd D \cap H_2$ is contained
in $H'_2$.
 $\bdd$-compressing a copy of $A_1$ along $D$,
we obtain a compressing disk $D_2$ of $H'_2$ in $W_2$.

 The annulus $A_2$ is $t_1$-compressible or $t_1$-$\bdd$-compressible
in $(V_1, t_1)$ by Lemma 2.4 in \cite{GH2}.
 First we consider the former case.
 Let $R$ be a $t_1$-compressing disk of $A_2$.
 If $R$ is contained in $W_1$,
then $R$ and $D_2$ together show
that $H_2$ is weakly $K$-reducible.
 If $R$ is contained in $W_2$,
then by compressing a copy of $A_2$ along $R$
we obtain a compressing disk of $A_1$ in $W_2$.
 This contradicts our assumption.

 Hence we may assume
that $A_2$ has a $t_1$-$\bdd$-compressing disk $C$ in $(V_1, t_1)$.
 This disk $C$ is not contained in $W_1$
since the two boundary loops $\bdd A_2$
are contained in distinct components of $H_1 \cap W_1 = Q \cup H'_1$.
 Then it is contained in $W_2$,
and incident to $A_1$.
 Hence the annuli $A_1$ and $A_2$ are parallel in $W_2$.
 This again contradicts our assumption.
\end{proof}

 In other cases, the arguments are similar to those in the previous section.

\begin{lemma}\label{lem:1b-A1compressible}
 Suppose that $A_1$ is compressible in $W_2$,
and that $W_1$ contains either
$(0)$ a $t_1$-compressing disk of $A_2$ in $(V_1, t_1)$, 
$(1)$ a $t_2$-compressing disk of $H'_2$ in $(V_2, t_2)$,
$(2)$ a $t_2$-$\bdd$-compressing disk of $H'_2$
incident to $Q$ in $(V_2, t_2)$,
$(3)$ a $K$-compressing disk of $Q \cup H'_1$ in $(W_1, K)$
or 
$(4)$ a $K$-$\bdd$-compressing disk of $Q \cup H'_1$ 
incident to $Q$ in $(W_1, K)$.
 Then $H_2$ is weakly $K$-reducible.
\end{lemma}

\begin{proof}
 Compressing $A_1$,
we obtain a disk $D_2\,(\subset W_2)$ bounded by $l_1$.

 In Case (2), 
the $t_2$-$\bdd$-compressing disk is also 
a $K$-$\bdd$-compressing disk of $Q$
by the definition of $\bdd$-compressing disk.
 Hence, in Cases (2), (3) and (4),
we obtain a $K$-compressing or meridionally compressing disk
of $H'_2$ in $(W_1, K)$
by performing a compressing or $\bdd$-compressing operation
on a copy of $Q$ or $H'_1$.
 Hence, in all cases, there is 
a $K$-compressing or meridionally compressing disk $D_1$ 
of $H'_2$ or $A_2$ in $(W_1, K)$.
 Then the disks $D_1$ and $D_2$ together show
that $H_2$ is weakly $K$-reducible.
\end{proof}

\begin{lemma}\label{lem:1b-A1parallel}
 Suppose that $A_1$ is parallel to $A_2$ in $W_2$,
and that $W_1$ contains either
$(1)$ a $t_2$-compressing disk of $H'_2$,
$(2)$ a $t_2$-$\bdd$-compressing disk of $H'_2$
incident to $Q$,
$(3)$$'$ a $K$-compressing disk of $Q \cup H'_1$ in $V_2$
or 
$(4)$ a $K$-$\bdd$-compressing disk of $Q \cup H'_1$ incident to $Q$.
 Then $H_2$ is weakly $K$-reducible.
\end{lemma}

\begin{proof}
 In Case (4),
the $K$-$\bdd$-compressing disk is incident to $H'_2$ rather than $A_2$
by the definition of $\bdd$-compressing disk.
 Thus, in all cases,
we can assume that $V_2 \cap W_1$ contains
a $K$-compressing or meridionally compressing disk $D$ of $H'_2$
in $(W_1, K)$
by a similar argument to in the second paragraph 
in the proof of Lemma \ref{lem:1b-A1compressible}.

 Since $A_1$ is parallel to $A_2$ in $W_2$,
we can isotope $H_1$ along the parallelism
so that $H_1$ is contained in int\,$W_1$ and is disjoint from $D$.
 Then Proposition \ref{thm:koba} shows
that $H_2$ is weakly $K$-reducible.
\end{proof}

 We call a $(2,0)$-splitting
$(M, K) = (W_1, K) \cup_{H_2} (W_2, \emptyset)$
{\it $K$-stabilized\/}
if there is a $K$-compressing disk of $H_2$ in $(W_i, K \cap W_i)$
for $i=1$ and $2$
such that $\bdd D_1$ and $\bdd D_2$ intersect each other
transversely in a single point.
 A $K$-stabilized $(2,0)$-splitting is $K$-reducible.
 See, for example, Lemma 4.1 in \cite{GHY}.
 (See Definition 4.1 in \cite{GH2} for the definition of $K$-reducibility.)

\begin{lemma}\label{lem:1b-H'_1H'_2}
 Suppose that $A_1$ is compressible 
or parallel to $A_2$ in $W_2$.
 Assume that $V_2 \cap W_1$ contains
a $t_2$-$\bdd$-compressing disk of $H'_2$ incident to $H'_1$
or a $K$-$\bdd$-compressing disk of $Q \cup H'_1$ incident to $H'_1$.
 Then one of the following conditions holds.
\begin{enumerate}
\renewcommand{\labelenumi}{(\theenumi)}
\item
 $H_2$ is weakly $K$-reducible.
\item
 $H_2$ is semi-stabilized,
and $K$ is a torus knot or a satellite knot.
\end{enumerate}
 Moreover,
if either $A_1$ is compressible in $W_2$
or $Q \cup H'_1$ is $K$-incompressible, 
then the conclusion $(1)$ holds.
\end{lemma}

\begin{proof}
 We isotope $H_1$ along the $\bdd$-compressing disk $D$
slightly beyond the arc $\bdd D \cap H'_2$.
 Then, 
by the definition of $\bdd$-compressing disk,
$H'_1$ is deformed into an annulus $A$,
each of the boundary loops $\bdd A$ is essential in $H'_2$.
 The annulus $A_1$ is deformed to a disk with two holes $P$.
 The annulus $A_2$ is deformed to a disk with two holes $P_{21}$.
 The torus with two holes $H'_2$ is deformed to a $2$-manifold $H''_2$,
which is either a disk with two holes $P_{22}$,
or a disjoint union of an annulus $A'$
and a torus with one hole $H^*_2$.
 Note that a component of $\bdd A'$ is $l_1 = \bdd Q$.
 See Figure \ref{fig:11-1}.
 Then $P$ is parallel to $P_{21}$
if $A_1$ is parallel to $A_2$ in $W_2$ before the isotopy.
 By Lemma 3.1 in \cite{GH2}, 
$Q \cup A$ is $K$-compressible or $K$-$\bdd$-compressible
in $(W_1, K)$.

\begin{figure}[htbp]
\centering
\includegraphics[width=.7\textwidth]{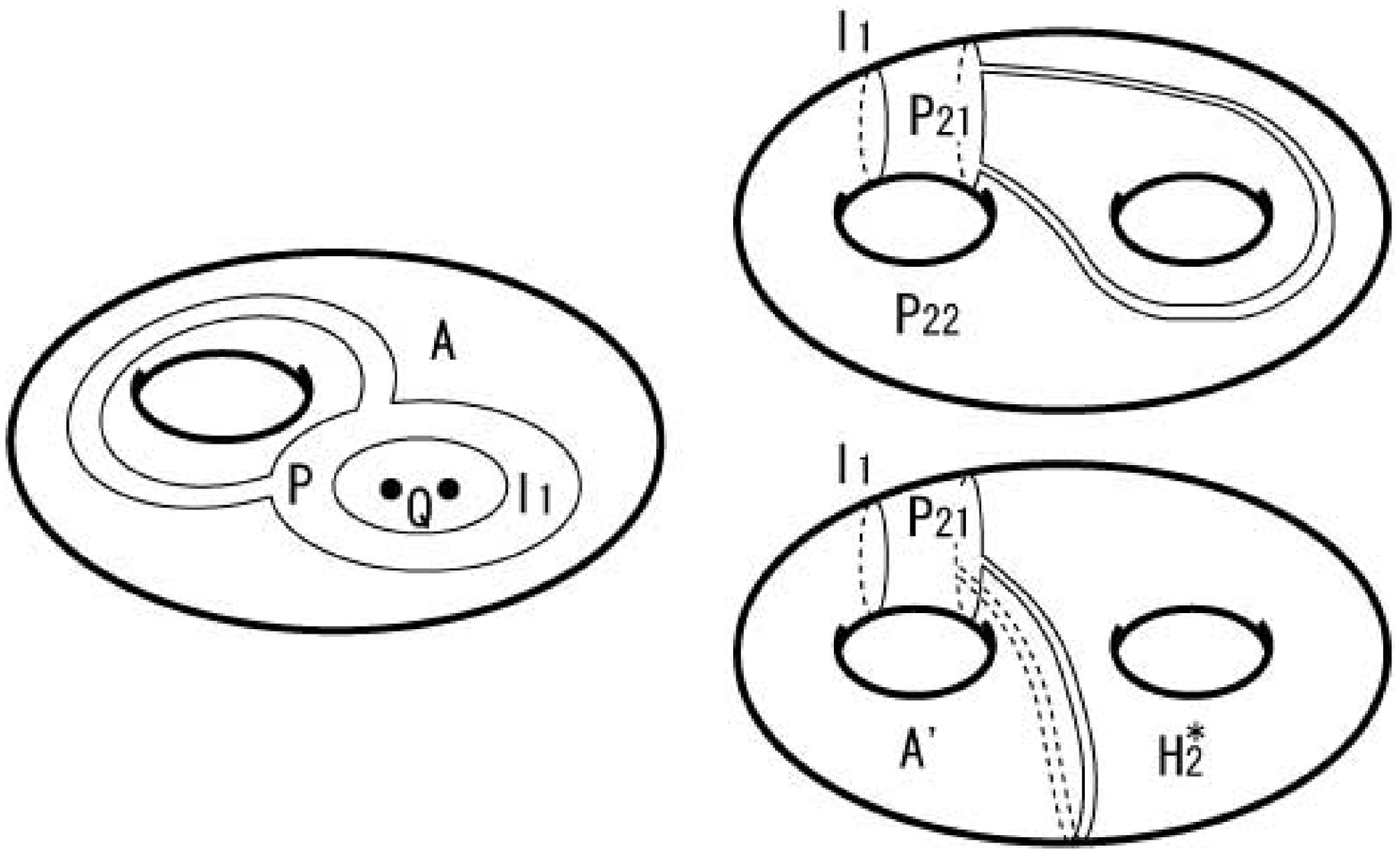}
\caption{}
\label{fig:11-1}
\end{figure}

{\it Case $(a)$.}
 First, we assume
that $Q \cup A$ is $K$-incompressible
and has a $K$-$\bdd$-compressing disk $R$ in $(W_1, K)$.
 (This holds if $Q \cup H'_1$ is $K$-incompressible before the isotopy.)
 The arc $\bdd R \cap H_2$ is not contained in the annulus $A'$
because of the definition of $\bdd$-compressing disk.
 The arc $\bdd R \cap H_2$ is not contained in $H^*_2$
since it contains only one component of $\bdd A$.
 Thus the arc $\bdd R \cap H_2$ is contained
in a disk with two holes $P_{21}$ or $P_{22}$.

 If the arc $\bdd R \cap H_2$ is contained in $P_{22}$,
then by $\bdd$-compressing $Q \cup A$ along $R$
we obtain an essential disk $R_1$ in $W_1$
such that it intersects $K$ in at most one point.
 When $A_1$ is parallel to $A_2$ in $W_2$ before the isotopy,
$P$ is parallel to $P_{21}$ in $W_2$ after the isotopy, 
and hence we can isotope $H_1$ in int\,$W_1$
so that it is disjoint from $R_1$.
 Then Proposition \ref{thm:koba} shows
that $H_2$ is weakly $K$-reducible.
 When $A_1$ is compressible in $W_2$,
a compressing operation on $A_1$ yields a disk bounded by $l_1$, 
which shows together with $R_1$
that $H_2$ is weakly $K$-reducible.

 Hence we may assume
that the arc $\bdd R \cap H_2$ is contained in $P_{21}$.
 If $R$ is incident to $A$,
then performing a $K$-$\bdd$-compressing operation on $A$ along $R$,
we obtain a disk $R_2$, disjoint from $K$.
 Since the arc $\bdd R \cap H_2$ is an essential arc in $P_{21}$
and since it connects the two loops $\bdd A'$,
the boundary loop $\bdd R_2$ is parallel to $l_1 = \bdd Q$.
 This implies that $Q$ is $K$-compressible in $(W_1, K)$,
contradicting our assumption.
 Hence $R$ is incident to $Q$.
 Compressing $Q$ along $R$,
we obtain two disks $R_3$ and $R_4$,
each of which intersects $K$ transversely in a single point.
 We can isotope $R_3$ and $R_4$ in $(W_1, K)$
so that they are bounded by the two loops $\bdd A$.
 When $H''_2 = A' \cup H^*_2$,
one of the disks $R_3$ and $R_4$ is bounded by $\bdd H^*_2$,
and hence is separating in $W_1$.
 This contradicts that each of $R_3$ and $R_4$ 
intersects $K$ transversely in a single point.
 Hence we can assume that $H''_2 = P_{22}$.
 When $A_1$ is compressible in $W_2$,
by compressing $A_1$, 
we obtain an essential disk bounded by $l_1$ in $W_2$. 
 This disk and $R_3$ show
that $H_2$ is weakly $K$-reducible.
 We consider the case where $A_1$ is parallel to $A_2$ in $W_2$ 
before the isotopy.
 In this case, $P$ is parallel to $P_{21}$ in $W_2$ after the isotopy.
 By Lemma 2.4 in \cite{GH2}, 
$P_{22}$ is $t_2$-compressible or $t_2$-$\bdd$-compressible
in $(V_2, t_2)$.
 First we consider the former case.
 If the $t_2$-compressing disk of $P_{22}$ is in $W_1$,
then the same argument as in the third paragraph in this proof
shows that $H_2$ is weakly $K$-reducible.
 If the $t_2$-compressing disk of $P_{22}$ is in $W_2$,
then this disk together with $R_3$ shows 
that $H_2$ is weakly $K$-reducible.

 We consider the latter case,
where $P_{22}$ has a $t_2$-$\bdd$-compressing disk $C$.
 If $C$ is contained in $W_1$,
then it is also a $K$-$\bdd$-compressing disk of $Q \cup A$.
 We have considered this situation
in the third paragraph of this proof.
 Hence we may assume that $C$ is contained in $W_2$.
 The loop $\bdd C$ intersects 
one of the loops $\bdd R_3$ and $\bdd R_4$, say $\bdd R_3$,
in at most one point.
 Since $P$ is parallel to $P_{21}$ in $W_2$,
and we can extend $C$ to an essential disk in $W_2$
such that its boundary loop intersects $\bdd R_3$
in at most one point.
 Hence $H_2$ is weakly $K$-reducible or meridionally stabilized.
 Also in the latter case,
$H_2$ is weakly $K$-reducible 
by Proposition 4.7 in \cite{GH2}.

{\it Case $(b)$.}
 We consider the case 
where $Q \cup A$ is $K$-compressible in $(W_1, K)$.
 Then $Q \cup H'_1$ is $K$-compressible before the isotopy.
 If the compressing disk is in $V_2$,
then Lemmas \ref{lem:1b-A1compressible} and \ref{lem:1b-A1parallel} 
show that $H_2$ is weakly $K$-reducible.
 Hence we can assume that the $K$-compressing disk is in $V_1$,
and a compressing operation on $Q \cup H'_1$
yields a $K$-compressing disk $X$ of $A_2$ in $V_1 \cap W_1$.
 We can isotope so that $\bdd X = l_1$.
 When $A_1$ is compressible in $W_2$,
a compressing operation on $A_1$ yields a disk bounded by $l_2$,
which together with $X$ shows
that $H_2$ is weakly $K$-reducible.

 Hence we can assume that $A_1$ is parallel to $A_2$ in $W_2$.
 In this case, $P$ is parallel to $P_{21}$ in $W_2$.
 $H''_2$ is $t_2$-compressible or $t_2$-$\bdd$-compressible
in $(V_2, t_2)$.
 Suppose that $H''_2$ is $t_2$-compressible.
 If the $t_2$-compressing disk is in $W_2$,
then this disk and $X$ show
that $H_2$ is weakly $K$-reducible.
 If it is in $W_1$,
then $H'_2$ has a $K$-compressing disk in $(W_1, K)$ before the isotopy,
and Lemmas \ref{lem:1b-A1compressible} and \ref{lem:1b-A1parallel}
show that $H_2$ is weakly $K$-reducible.
 Hence we can assume
$H''_2$ is $t_2$-incompressible
and has a $t_2$-$\bdd$-compressing disk $Z$ in $(V_2, t_2)$.
 First suppose that $Z$ is contained in $W_1$.
 Then $Z$ is not incident to the annulus $A'$
because the two boundary loops of $\bdd A'$ are contained
in distinct components of $H_1 \cap W_1 = Q \cup A$
separately.
 Moreover, $C$ is not incident to the torus with one hole $H^*_2$
because it contains only one component of $\bdd A$.
 Hence $H''_2 = P_{22}$.
 Performing the $K$-$\bdd$-compressing operation on $Q \cup A$
along the disk $Z$,
we obtain a disk $Z_1$
which intersects $K$ in at most one point.
 Note that $\bdd Z_1$ is essential in $H_2$
since $Z$ is a $t_2$-$\bdd$-compressing disk of $P_{22}$.
 Along the parallelism of $P$ and $P_{21}$,
we can isotope $H_1$
so that $P$ is pushed into $\text{int}\,W_1$
and that $H_1 \cap Z_1 = \emptyset$.
 Then Proposition \ref{thm:koba} shows
that $H_2$ is weakly $K$-reducible.
 This is the conclusion (1).

 Therefore, we may assume that the $t_2$-$\bdd$-compressing disk $Z$
of $H''_2$ is contained in $W_2$.
 We can extend $Z$ into an essential disk $Z'$ in $W_2$
because $P$ and $P_{21}$ are parallel in $W_2$.
 Since $\bdd Z$ intersects the loop $l_1$ at most in two points,
so does $\bdd Z'$. 
 Hence the disks $Z'$ and $X$ show 
that $H_2$ is either weakly $K$-reducible, $K$-stabilized or 
semi-stabilized. 
 In the second case,
$H_2$ is weakly $K$-reducible.
 In the last case,
we have the conclusion (3) by Proposition \ref{prop:semistabilized}.
\end{proof}

\begin{lemma}\label{lem:1b-A1compressible2}
 Suppose that $A_1$ is compressible in $W_2$.
 Then $H_2$ is weakly $K$-reducible.
\end{lemma}

\begin{proof}
 $Q \cup H'_1$ is $K$-compressible or $K$-$\bdd$-compressible in $(W_1, K)$ 
by Lemma 3.1 in \cite{GH2}.
 In the former case, we have the conclusion 
by Lemma \ref{lem:1b-A1compressible}.
 In the latter case, 
the $K$-$\bdd$-compressing disk is not incident to $A_2$ by the definition.
 Hence the conclusion holds 
by Lemmas \ref{lem:1b-A1compressible} and \ref{lem:1b-H'_1H'_2}.
\end{proof}

\begin{lemma}\label{lem:1b-QHcompressible}
 Suppose that $A_1$ is parallel to $A_2$ in $W_2$,
and that $Q \cup H'_1$ is $K$-compressible in $(W_1, K)$.
 Then one of the following conditions holds.
\begin{enumerate}
\renewcommand{\labelenumi}{(\theenumi)}
\item
 The $(2,0)$-splitting $H_2$ of $(M,K)$ is weakly $K$-reducible.
\item
$H_2$ is semi-stabilized,
and $K$ is a torus knot or a satellite knot.
\end{enumerate}
\end{lemma}

\begin{proof}
 Let $D$ be a $K$-compressing disk of $Q \cup H'_1$ in $(W_1, K)$.
 If $D$ is in $V_2$,
then Lemmas \ref{lem:1b-A1compressible} and \ref{lem:1b-A1parallel} 
show that $H_2$ is weakly $K$-reducible.
 Thus we may assume that $D$ is in $V_1$.
 Compressing a copy of $Q$ or $H'_1$ along $D$,
we obtain a disk $D_1$ which is disjoint from $K$.
 We can isotope $D_1$ in $(W_1, K)$
so that $D_1 \cap (Q \cup H'_1) = \bdd D_1 \cap \bdd Q = l_1$.
 Then $D_1$ forms a $K$-compressing disk of $A_2$.

 $H'_2$ is $t_2$-compressible or $t_2$-$\bdd$-compressible in $(V_2, t_2)$
by Lemma 2.4 in \cite{GH2}.
 In the former case,
let $R$ be a $t_2$-compressing disk of $H'_2$.
 If $R$ is contained in $W_1$, 
then Lemma \ref{lem:1b-A1parallel} show
that $H_2$ is weakly $K$-reducible.
 If $R$ is contained in $W_2$,
the disks $D_1$ and $R$ together show
that $H_2$ is weakly $K$-reducible.

 Hence we may assume
that $H'_2$ is $t_2$-incompressible
and has a $t_2$-$\bdd$-compressing disk $R'$ in $(V_2, t_2)$.
 We first consider the case
where $R'$ is contained in $W_2$.
 Since $H_1 \cap W_2 =A_1$ is an annulus disjoint from $K$,
the loop $\bdd R'$ intersects each of the loop components of $\bdd A_1$
transversely in a single point.
 Because $A_1$ is parallel to $A_2$ in $W_2$,
we can extend $R'$ to an essential disk in $W_2$
such that its boundary loop intersects the loop $l_1 = \bdd D_1$
transversely in a single point.
 This disk and $D_1$ show that $H_2$ is $K$-stabilized,
and we obtain the conclusion (1).
 Hence we may assume that the disk $R'$ is contained in $W_1$.
 Then we obtain the conclusion
by Lemmas \ref{lem:1b-A1parallel} and \ref{lem:1b-H'_1H'_2}.
\end{proof}

\begin{lemma}\label{lem:1b-QHincompressible}
 Suppose that $A_1$ is parallel to $A_2$ in $W_2$, 
and that $Q \cup H'_1$ is $K$-incompressible in $(W_1, K)$.
 Then the $(2,0)$-splitting $H_2$ of $(M,K)$ is weakly $K$-reducible.
\end{lemma}

\begin{proof}
 Since $Q \cup H'_1$ is $K$-incompressible,
it has a $K$-$\bdd$-compressing disk $D$ in $(W_1, K))$.
 By the definition of $K$-$\bdd$-compressing disk,
the arc $\bdd D \cap H_2$ is contained in $H'_2$
rather than $A_2$,
since one component of $\bdd A_2$ bounds $Q$
and the other bounds $H'_1$.
 Then Lemmas \ref{lem:1b-A1parallel} and \ref{lem:1b-H'_1H'_2}
show that $H_2$ is weakly $K$-reducible.
\end{proof}

 Thus, in Case (1)(B), we have
Conclusion (b), (c) or (d) of Theorem \ref{thm:l=2}
by Lemmas \ref{lem:1b-incompressible2}, 
\ref{lem:1b-A1compressible2}, \ref{lem:1b-QHcompressible}
and \ref{lem:1b-QHincompressible}.


\section{Case (2)(A)}\label{sect:2A}

 We consider in this section the case
where the loops $H_1 \cap H_2$ are
of the pattern (2) in $H_1$ and of the pattern (A) in $H_2$.
(See Figure \ref{fig:l=2}.)

 In Case (2),
the loops $l_1$ and $l_2$ of $H_1 \cap H_2$
together separate the torus $H_1$
into two annuli $A_{11}$ and $A_{12}$,
where $A_{1i}$ is contained in the handlebody $W_i$ for $i=1$ and $2$.
 Note that the two intersection points $K \cap H_1$
are contained in $A_{11}$
since the knot $K$ is entirely contained in $W_1$.
 Let $A_2, H_{21}, H_{22}$ be
as in Section \ref{sect:1A}.
 We assume, without loss of generality,
that $A_2$ is contained in $V_1$,
and $H_{21}\cup H_{22}$ in $V_2$.

 By Lemma 3.1 in \cite{GH2},
$A_{11}$ is $K$-compressible or $K$-$\bdd$-compressible
in $(W_1, K)$.
$A_{12}$ is compressible or $\bdd$-compressible
in the handlebody $W_2$.

\begin{lemma}\label{lem:2a-A_12compressible}
 Suppose that $A_{12}$ is compressible in $W_2$.
 Then $H_2$ is weakly $K$-reducible.
\end{lemma}

\begin{proof}
 Let $D$ be a compressing disk of $A_{12}$ in $W_2$.
 Compressing $A_{12}$ along this disk $D$,
we obtain two disks $D_1$ and $D_2$
bounded by $l_1$ and $l_2$ respectively.
 Suppose that $D$ is contained in $V_1$.
 Then the disk $D_1$ is also contained in $V_1$,
and we can isotope it slightly into $\text{int}\,V_1$
so that $\bdd D_1 \subset \text{int}\,A_2$
and $D_1 \cap A_{12} = \emptyset$.
 The disk $D_1$ is separating in $W_2$
and separates the loops $l_1$ and $l_2$.
 This contradicts
that the annulus $A_{12}$ connects these loops
and is disjoint from $D_1$.
 Hence the disks $D$, $D_1$ and $D_2$
are contained in $V_2$.

 Suppose first
that $A_{11}$ is $K$-compressible in $(W_1, K)$.
 Let $R$ be a $K$-compressing disk of $A_{11}$.
 If $\bdd R$ is essential in $A_{11}$
ignoring the intersection points with $K$,
then, performing a $K$-compressing operation on $A_{11}$,
we obtain a disk $R_1$ bounded by $l_1$ or $l_2$
such that $R_1$ intersects $K$ transversely in at most one point.
 Then the disks $D_1$ and $R_1$ together show
that $H_2$ is weakly $K$-reducible.

 Hence we may assume
that $\bdd R$ is inessential in the annulus $A_{11}$.
 Suppose that $R$ is contained in $V_1$.
 Then the disks $R$ and $D$ show
that $H_1$ is $K$-reducible.
 This implies that $K$ is the trivial knot 
by Proposition 4.3 in \cite{GH2}
(Lemma 2.10 in \cite{Hy3}), 
and hence $H_2$ is (weakly) $K$-reducible 
by Proposition 4.6 in \cite{GH2}
(Proposition 2.9 in \cite{GHY}).
 Hence we may assume
that the disk $R$ is contained in $V_2$.
 Compressing $A_{11}$ along $R$,
we obtain an annulus $A$ disjoint from $K$.
 Note that $\bdd A = \bdd A_{11} = l_1 \cup l_2$,
and that $K$ is entirely contained in the $3$-manifold
between $A_2$ and $A$ in $W_1$.
 The annulus $A$ is $K$-compressible or $K$-$\bdd$-compressible
in $(W_1, K)$ by Lemma 3.1 in \cite{GH2}.
 In the former case, compressing $A$,
we obtain two disks
which are disjoint from $K$
and bounded by the loops $l_1$ and $l_2$.
 Then the union of these disks and the annulus $A_2$ forms a $2$-sphere,
which bounds a $3$-ball $B$ in $W_1$
such that $B$ entirely contains $K$.
 This contradicts that $K$ is a core in $W_1$. 
 In the latter case,
let $R'$ be a $K$-$\bdd$-compressing disk of $A$.
 Since the two loops of $\bdd A$ are contained
in distinct components of $H_2 \cap V_2 = H_{21} \cup H_{22}$,
$R'$ is not contained in $V_2$.
 Hence $R'$ is contained in $V_1$,
and the arc $\bdd R' \cap H_2$ is contained in $A_2$.
 Performing a $K$-$\bdd$-compressing operation on the annulus $A$
along this disk $R'$,
we obtain a peripheral disk
which cuts off a $3$-ball containing $K$ from $W_1$.
 This again contradicts
that $K$ forms a core of $W_1$.

 Hence we may assume
that the annulus $A_{11}$ is $K$-incompressible,
and then it has a $K$-$\bdd$-compressing disk $C$ in $(W_1, K)$.
 Suppose first that $C$ is contained in $V_1$.
 Then, by the definition of $K$-$\bdd$-compressing disk,
$\bdd C$ intersects the annulus $A_2$ in an essential arc,
and $C$ forms a $t_1$-$\bdd$-compressing disk of $A_2$ in $(V_1, t_1)$.
 Performing a $\bdd$-compressing operation on a copy of $A_2$ along $C$,
we obtain a $K$-compressing disk of $A_{11}$.
 This contradicts our assumption.
 Hence we may assume that $C$ is contained in $V_2$.
 Since the two boundary loops of $\bdd A_{11}$ are contained
in distinct components of $H_2 \cap V_2 = H_{21} \cup H_{22}$,
$\bdd C$ intersects the annulus $A_{11}$ in an arc
which is inessential on $A_{11}$
ignoring the intersection points with $K$.
 Performing a $\bdd$-compressing operation $A_{11}$ along $C$,
we obtain an annulus $Z$ and a disk $P$.
 Note that a component of $\bdd Z$ is $l_1$ or $l_2$, say $l_1$,
and hence the other components of $\bdd Z$ and $\bdd P$
are disjoint from the loop $l_1= \bdd D_1$.
 The disk $P$ intersects $K$ transversely
in one or two points.
 When it intersects $K$ in one point,
it forms a meridionally compressing disk of $H_2$ in $(W_1, K)$.
 Hence the disks $D_1$ and $P$ together show
that $H_2$ is weakly $K$-reducible.
 When $P$ intersects $K$ in two points,
the annulus $Z$ is disjoint form $K$.
 Then $Z$ is $K$-compressible or $K$-$\bdd$-compressible
in $(W_1, K)$ by Lemma 3.1 in \cite{GH2}.
 In the former case, compressing $Z$,
we obtain a $K$-compressing disk of $H_2$ in $(W_1, K)$
such that it is bounded by $l_1$.
 Hence this disk and $D_1$ together show
that $H_2$ is weakly $K$-reducible.
 In the latter case,
by performing a $\bdd$-compressing operation on $Z$
and isotoping the resulting disk slightly,
we obtain a $K$-compressing disk of $H_2$ in $(W_1, K)$
such that its boundary loop is disjoint from the loop $l_1$.
 Note that loops of $\bdd Z$ are not parallel in $H_2$. 
 Hence this disk and $D_1$ together show
that $H_2$ is weakly $K$-reducible.
\end{proof}

\begin{lemma}\label{lem:2a-A_12incompressible}
 Suppose that $A_{12}$ is incompressible in $W_2$.
 Then $H_2$ is weakly $K$-reducible.
\end{lemma}


\begin{proof}
 Since $A_{12}$ is incompressible in $W_2$,
it is $\bdd$-compressible.
 Then $A_{12}$ is parallel to $A_2$ in $W_2$.

 By Lemma 2.4 in \cite{GH2},
$H_{21} \cup H_{22}$
is $t_2$-compressible or $t_2$-$\bdd$-compressible in $(V_2, t_2)$.
 We consider first the former case.
 Let $D$ be a $t_2$-compressing disk of $H_{21} \cup H_{22}$.
 When $D$ is in $W_2$,
compressing $H_{21} \cup H_{22}$,
we obtain a compressing disk of $A_{12}$.
 This contradicts our assumption.
 When $D$ is in $W_1$,
we can isotope the torus $H_1$ into $\text{int}\,W_1$
so that $H_1 \cap D = \emptyset$.
 Then Proposition \ref{thm:koba} shows
that $H_2$ is weakly $K$-reducible.

 Hence we may assume
that $H_{21} \cup H_{22}$ has a $t_2$-$\bdd$-compressing disk $R$.
 $R$ cannot be in $W_2$,
since $A_{12}$ cannot contain the arc $\bdd R \cap H_1$
by the definition of $\bdd$-compressing disk.
 Hence $R$ is contained in $W_1$,
and is a $K$-$\bdd$-compressing disk of the annulus $A_{11}$
by the definition of $t_2$-$\bdd$-compressing disk again.
 We assume, without loss of generality,
that $R$ is incident to $H_{22}$ rather than $H_{21}$.
 We isotope the torus $H_1$ in $(M, K)$ along the disk $R$
slightly beyond the arc $\bdd R \cap H_{22}$.
 Then the annulus $A_{11}$ is deformed into an annulus, say $A$,
and a disk, say $R_1$.
 The torus with one hole $H_{22}$ is deformed
into an annulus $A'$.
 The annuli $A_{12}$ and $A_2$ are deformed
into disks with two holes, say $P_1$ and $P_2$ respectively.
 See Figure \ref{fig:12-2}.
 Note that $P_1$ is parallel to $P_2$ in $W_2$
since $A_{12}$ is parallel to $A_2$ in $W_2$ before the isotopy.
 The disk $R_1$ intersects $K$
transversely in one or two points.
 When it intersects $K$ in one point,
we can isotope the torus $H_1$ into $\text{int}\,W_1$
along the parallelism between $P_1$ and $P_2$.
 Further, we can take a parallel copy of $R_1$
so that it is disjoint from $H_1$.
 Then Proposition \ref{thm:koba} shows
that $H_2$ is weakly $K$-reducible.

\begin{figure}[htbp]
\centering
\includegraphics[width=.7\textwidth]{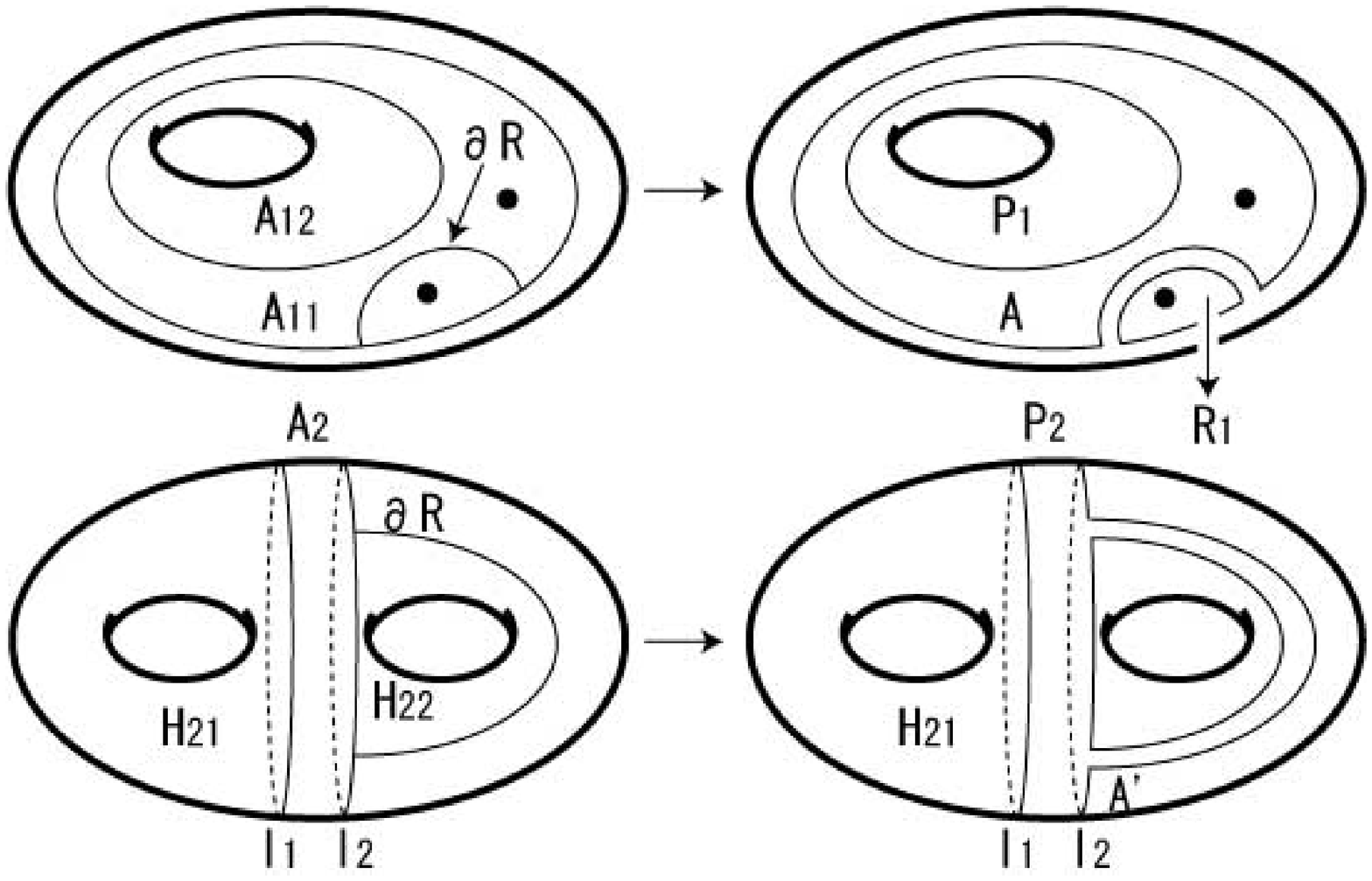}
\caption{}
\label{fig:12-2}
\end{figure}

 Hence we may assume
that the disk $R_1$ intersects $K$ in two points.
 Let $R_2$ be a disk bounded by the loop $l_1$ in $W_1$
such that it is obtained from the disk $A \cup A' \cup R_1$
by pushing its interior into $\text{int}\,V_2 \cap W_1$.
 Then this disk $R_2$ is bounded by the loop $l_1$,
and intersects $K$ transversely in two points.
 Moreover, $R_2$ divides the handlebody $W_1$
into two solid tours $U_1$ and $U_2$
where $U_1$ is bounded by the torus $H_{21} \cup R_2$.
 By Lemma 3.1 in \cite{GH2},
$R_2$ is $K$-compressible or $K$-$\bdd$-compressible
in $(W_1, K)$.
 We consider first the latter case.
 Let $E$ be a $K$-$\bdd$-compressing disk of $R_2$.
 When $E$ is contained in $U_1$,
performing a $K$-$\bdd$-compressing on $R_2$ along $E$,
we obtain a disk $E_1$ intersecting $K$ in a single point.
 We isotope $E_1$ slightly off of the disk $R_2$
so that $\bdd E_1$ is in $H_{21}$.
 We can isotope the torus $H_1$ into $\text{int}\,W_1$
along the parallelism between $P_1$ and $P_2$
so that $H_1 \cap E_1 = \emptyset$.
 Then Proposition \ref{thm:koba} shows
that $H_2$ is weakly $K$-reducible.
 When $E$ is contained in the other solid torus $U_2$,
by performing a $K$-$\bdd$-compressing operation on $R_2$,
we obtain a meridian disk, say $E_2$,
intersecting $K$ in a single point.
 We can form a knot $K'$ taking a sum of the arc $K \cap U_2$
and an arc connecting the two points $K \cap R_2$ in $R_2$.
 Thus the disk $R_1$ intersects $K'$
transversely in precisely two points,
while the disk $E_2$ intersects $K'$
transversely in a single point.
 This contradicts the fact
that $R_1$ and $E_2$ represent the same homology class
in $H_2 (U_2 ; \bdd U_2)$,
and hence they have the same algebraic intersection number
with $[K'] \in H_1 (U_2)$.

 Hence we may assume that $R_2$ has a $K$-compressing disk
in $(W_1,K)$. 
 Performing a $K$-compressing operation on $R_2$,
we obtain a disk, say $G$, bounded by the loop $l_1$.
 If $G$ is contained in $U_2$,
then it separates the intersection points $K \cap R_2$
from $K \cap (A_{11} \cup R_1)$,
a contradiction.
 Hence $G$ is contained in $U_1$.
 Then we move $H_1$ into int\,$W_1$ 
along the parallelism between $P_1$ and $P_2$
so that $H_1 \cap G = \emptyset$,
to see that $H_2$ is weakly $K$-reducible
by Proposition \ref{thm:koba}.
\end{proof}

 Thus, in Case 2(A), we have the conclusion (b) of Theorem \ref{thm:l=2}
by Lemmas \ref{lem:2a-A_12compressible} and \ref{lem:2a-A_12incompressible}.


\section{Case (2)(B)}\label{sect:2B}

 We consider in this section the case
where the intersection loops $H_1 \cap H_2 = l_1 \cup l_2$
are of the pattern (2) in $H_1$
and of the pattern (B) in $H_2$  in Figure \ref{fig:l=2}.

Let $A_2, H'_2$ be as in Section \ref{sect:1B},
and $A_{11}$, $A_{12}$ as in Section \ref{sect:2A}.
 See Figure \ref{fig:l=2}.
 We may assume, without loss of generality,
that $A_2, H'_2$ are properly embedded
in $V_1, V_2$ respectively.
 Note that $A_{11}, A_{12}$ are properly embedded in $W_1, W_2$ respectively,
and the two intersection points $K \cap H_1$
are contained in the annulus $A_{11}$
since the knot $K$ is entirely contained in $W_1$.

$A_{12}$ is compressible or $\bdd$-compressible
in the handlebody $W_2$,
so we have four cases below.
\begin{enumerate}
\renewcommand{\labelenumi}{(\roman{enumi})}
\item
$A_{12}$ has a compressing disk in $V_1$.
\item
$A_{12}$ has a compressing disk in $V_2$.
\item
$A_{12}$ has a $\bdd$-compressing disk in $V_1$.
\item
$A_{12}$ has a $\bdd$-compressing disk in $V_2$.
\end{enumerate}

 By Lemma 3.1 in \cite{GH2},
$A_{11}$ is $K$-compressible or $K$-$\bdd$-compressible
in $(W_1, K)$.
 Here we divide into seven cases as below.

\begin{enumerate}
\renewcommand{\labelenumi}{(\Alph{enumi})}
\item
$A_{11}$ has a $K$-compressing disk
whose boundary loop is essential in $A_{11}$.
\item
In $V_1$, 
$A_{11}$ has a $K$-compressing disk
whose boundary loop is inessential in $A_{11}$.
\item
In $V_2$, 
$A_{11}$ has a $K$-compressing disk
whose boundary loop is inessential in $A_{11}$.
\item
In $V_1$, 
$A_{11}$ has a $K$-$\bdd$-compressing disk.
\item
In $V_2$, 
$A_{11}$ has a $K$-$\bdd$-compressing disk
whose boundary loop intersects $A_{11}$ in an essential arc.
\item
In $V_2$, 
$A_{11}$ has a $K$-$\bdd$-compressing disk
whose boundary loop intersects $A_{11}$ in an inessential arc
cutting off from $A_{11}$
a disk which intersects
$K$ in a single point.
\item
In $V_2$, 
$A_{11}$ has a $K$-$\bdd$-compressing disk
whose boundary loop intersects $A_{11}$ in an inessential arc
cutting off from $A_{11}$
a disk which intersects
$K$ in two points.
\end{enumerate}

%

 Hence we have $4 \times 7 = 28$ cases.
 By the next lemma,
we do not need to consider the 10 cases
(ii)(B), (ii)(C), (ii)(E), (ii)(F), (ii)(G),
(iv)(B), (iv)(C), (iv)(E), (iv)(F) and (iv)(G).

\begin{lemma}\label{lem:2b-(i)(iii)(A)or(D)}
 At least one of the four conditions $(i)$, $(iii)$, $(A)$ and $(D)$ holds.
\end{lemma}

\begin{proof}
 By Lemma 2.4 in \cite{GH2},
$A_2$ is $t_1$-compressible or $t_1$-$\bdd$-compressible
in $(V_1, t_1)$.
 In the former case,
let $D$ be a $t_1$-compressing disk of $A_2$.
 If $D$ is contained in $W_1$,
then compressing a copy of $A_2$ along $D$,
we obtain a $K$-compressing disk $D_1$ of $A_{11}$
such that $\bdd D_1$ is essential in $A_{11}$
ignoring the intersection points $K \cap A_{11}$.
 Thus the condition (A) holds.
 Hence we may assume that $D$ is contained in $W_2$.
 Compressing a copy of $A_2$ along $D$,
we obtain a compressing disk $D_2$ of $A_{12}$
such that $D_2$ is contained in $V_1$.
 Thus the condition (i) holds.

 In the latter case,
let $R$ be a $t_1$-$\bdd$-compressing disk of $A_2$.
 When $R$ is contained in $W_1$,
it is also a $K$-$\bdd$-compressing disk of $A_{11}$ in $(W_1, K)$
because of the definition of $t_1$-$\bdd$-compressing disk.
 Thus the condition (D) holds.
 Hence we may assume that $R$ is contained in $W_2$.
 Then $R$ is also a $\bdd$-compressing disk of $A_{12}$.
 Thus the condition (iii) holds.
\end{proof}


\begin{lemma}\label{lem:2b-D}
In Case $(D)$ the condition $(B)$ holds.
\end{lemma}

\begin{proof}
 In Case (D),
there is a $K$-$\bdd$-compressing disk $D$ of $A_{11}$
in $(W_1, K)$
such that $D$ is contained in $V_1$.
 Then $\bdd D \cap H_2$ is an essential arc in $A_2$
by the definition of $K$-$\bdd$-compressing disks.
 Hence $D$ is a $t_1$-$\bdd$-compressing disk of $A_2$
in $(V_1, t_1)$.
 Note that the arc $\bdd D \cap A_{11}$ is also essential in $A_{11}$
ignoring the intersection points $K \cap A_{11}$.
 Performing a $t_1$-$\bdd$-compressing operation on $A_2$ along $D$,
we obtain a $K$-compressing disk of $A_{11}$
such that its boundary loop is inessential in $A_{11}$
ignoring the intersection points $K \cap A_{11}$.
 Hence the condition (B) holds.
\end{proof}

We will show the present case of Theorem \ref{thm:l=2} 
in accordance with Table I.

\begin{table}[htbp]
\begin{center}
\begin{tabular}{|l|c|c|c|c|c|c|c|}
\hline
 & (A) & (B) & (C) & (D) & (E) & (F) & (G) \\
\hline
(i) & 5.8 & 5.12 & 5.4 & 5.12 & 5.3 & 5.5 & 5.5 \\
\hline
(ii) & 5.8 & --- & --- & 5.9 & --- & --- & --- \\
\hline
(iii) & 5.8 & 5.12 & 5.4 & 5.12 & 5.3 & 5.6 & 5.7 \\
\hline
(iv) & 5.8 & --- & --- & 5.10 & --- & --- & --- \\
\hline
\end{tabular}

\medskip

Table I
\end{center}
\end{table}

\begin{lemma}\label{lem:2b-E}
 In Case $(E)$
we can isotope $H_1$ in $(M,K)$
so that $H_1 \cap H_2$ consists of a single loop
which is $K$-essential both in $H_1$ and in $H_2$.
\end{lemma}

\begin{proof}
 In Case (E),
there is a $K$-$\bdd$-compressing disk $D$ of $A_{11}$
such that $D$ is contained in $V_2$
and such that the arc $\bdd D \cap A_{11}$ is essential in $A_{11}$
ignoring the intersection points with $K$.
 We isotope $H_1$ along $D$,
so that a band neighborhood of the arc $\bdd D \cap A_{11}$ in $A_{11}$
is isotoped into $W_2$.
 Then the annulus $A_{11}$ is deformed into a disk $Q$
intersecting $K$ in two points.
 Note that the boundary loop $\bdd Q$ is essential in $H_2$
since the arc $\bdd D \cap H'_2$ is essential in $H'_2$
by the definition of $K$-$\bdd$-compressing disks.
\end{proof}

\begin{lemma}\label{lem:2b-C}
 In Case $(C)$, the condition $(E)$ holds.
\end{lemma}

 We have already considered Case (E) in Lemma \ref{lem:2b-E}.

\begin{proof}
 In Case (C),
there is a $K$-compressing disk $D$ of $A_{11}$
such that $D$ is contained in $V_2$
and that $\bdd D$ bounds a disk $D'$
in $A_{11}$.
 Then $D'$ intersects $K$ in two points,
and a $K$-compressing operation on $A_{11}$ along $D$
yields an annulus $A$
such that it is disjoint from $K$
and that $\bdd A = \bdd A_{11}$.
 Since $A_{11}$ is separating in $W_1$, so is $A$.
 Note that $K$ is between the annuli $A$ and $A_2$.

 The annulus $A$ is $K$-compressible or $K$-$\bdd$-compressible
in $(W_1, K)$
by Lemma 3.1 in \cite{GH2}.
 In the former case,
performing a $K$-compressing operation on $A$,
we obtain two disks bounded by $H_1 \cap H_2 = l_1 \cup l_2$.
 The knot $K$ is in the ball between these disks,
which contradicts that $K$ is a core of $W_1$.
 In the latter case,
$A$ has a $K$-$\bdd$-compressing disk $R$.
 We assume first
that the arc $\bdd R \cap H_2$ is contained in $A_2$.
 Performing a $K$-$\bdd$-compressing operation on $A$ along $R$,
we obtain a peripheral disk
which cuts off a ball containing $K$ from $W_1$.
 This is again a contradiction.
 Hence $R$ is contained in $V_2$.
 We can isotope $R$
so that it is disjoint from the copy of $D$ in $A$.
 Then $R$ gives a $K$-$\bdd$-compressing disk of $A_{11}$.
 Because the arc $\bdd R \cap A$ is essential in $A$
so is the arc $\bdd R \cap A_{11}$ in $A_{11}$.
 Thus the condition (E) holds.
\end{proof}

\begin{lemma}\label{lem:2b-iFiiFiGiiG}
 In Cases $(i)(F)$ and
 $(i)(G)$,
$H_2$ is weakly $K$-reducible.
\end{lemma}

\begin{proof}
 In Case (i),
$A_{12}$ is compressible in $W_2$.
 Compressing $A_{12}$,
we obtain two disks $D_1$ and $D_2$ ($\subset V_1$)
bounded by the loops $l_1$ and $l_2$
respectively.

 In Case (F),
a $K$-$\bdd$-compressing operation on a copy of $A_{11}$ along $D$
yields a meridionally compressing disk $Q$ of $H_2$
in $(W_1, K)$.
 We can isotope $Q$ slightly off of $A_{11}$.
 Then the disks $D_1$ and $Q$ show
that $H_2$ is weakly $K$-reducible.

 In Case (G),
there is a $K$-$\bdd$-compressing disk $R$ of $A_{11}$ in $V_2$
such that the arc $\bdd R \cap A_{11}$ is inessential in $A_{11}$
and cuts off a disk $R'$ from $A_{11}$
and that $R'$ intersects $K$ in two points.
 By the definition of $K$-$\bdd$-compressing disks,
the arc $\bdd R \cap H_2$ is an essential arc in $H'_2$.
 Performing a $K$-$\bdd$-compressing operation on $A_{11}$ along $R$,
we obtain an annulus $A$ which is disjoint from $K$.
 Note that one component of $\bdd A$ is $l_1$ or $l_2$, say $l_1$,
and the other component is not parallel to $l_1$ in $H_2$.
 By Lemma 3.1 in \cite{GH2},
$A$ is $K$-compressible or $K$-$\bdd$-compressible in $(W_1, K)$.
 In the former case,
performing a $K$-compressing operation on $A$,
we obtain a disk $E_1$ bounded by $l_1$.
 Note that $E_1$ is disjoint form $K$.
 Then the disks $D_1$ and $E_1$ show
that $H_2$ is weakly $K$-reducible.
In the latter case,
performing a $K$-$\bdd$-compressing operation on $A$,
we obtain a disk $E_2$
such that it is disjoint form $K$
and that $\bdd E_2$ is essential in $H_2$.
 We can isotope $E_2$ slightly off of $A$,
and hence off of $l_1$.
 Then the disks $D_1$ and $E_2$ show
that $H_2$ is weakly $K$-reducible.
\end{proof}

\begin{lemma}\label{lem:2b-iiiF}
 In Case $(iii)(F)$,
$H_2$ is weakly $K$-reducible.
\end{lemma}

\begin{proof}
 In Case (iii),
$A_{12}$ is parallel to $A_2$ in $W_2$.
 In Case (F),
 A $K$-$\bdd$-compressing operation on $A_{11}$
yields a meridionally compressing disk $Q$ of $H_2$.
 We can isotope $Q$ slightly off of $A_{11}$.
 We isotope $H_1$ in $(M, K)$
along the parallelism between $A_{12}$ and $A_2$,
so that $H_1$ is contained in $W_1$
and that $H_1$ is disjoint from the disk $Q$.
 Then Proposition \ref{thm:koba} shows
that $H_2$ is weakly $K$-reducible.
\end{proof}

\begin{lemma}\label{lem:2b-iiiG}
 In Case $(iii)(G)$, 
we can isotope $H_1$ in $(M, K)$
so that $H_1$ and $H_2$ intersect each other
transversely in two loops
which are $K$-essential both in $H_1$ and $H_2$
and of the pattern $(1)$ in $H_1$  in Figure \ref{fig:l=2}.
\end{lemma}

\begin{proof}
 In Case (iii),
$A_{12}$ is parallel to $A_2$ in $W_2$.

In Case (G),
we isotope $H_1$ along the $K$-$\bdd$-compressing disk of $A_{11}$.
 Then the annulus $A_{11} = H_1 \cap W_1$ is deformed
into a disjoint union of an annulus $A$ and a disk $Q$
such that $A$ is disjoint from $K$
and $Q$ intersects $K$ in two points.
 Note that each loop of $\bdd Q$ and $\bdd A$ is essential in $H_2$.
 The annulus $A_{12} = H_1 \cap W_2$ is deformed
into a disk with two holes $P$
and also the annulus $A_2 = H_2 \cap V_1$
into a disk with two holes $P'$.
 Since $A_{12}$ is parallel to $A_2$ in $W_2$,
$P$ is parallel to $P'$ in $W_2$.

 There is a $\bdd$-compressing disk $R$ of $P$ in $W_2$
such that the arc $\bdd R \cap H_2$
is contained in $P'$ and connects the two boundary loops $\bdd A$.
 We further isotope $H_1$ along $R$.
 Then $P$ is deformed into an annulus,
and $A$ is deformed into a torus with one hole $T$.
 Note that $\bdd T$ is parallel to $\bdd Q$
since the arc $\bdd R \cap H_2$ is contained in $P'$.
 Thus we have isotoped $H_1$
so that $H_1$ and $H_2$ intersect each other
transversely in two loops
which are $K$-essential both in $H_1$ and in $H_2$
and of the pattern (1) in $H_1$ in Figure \ref{fig:l=2}.
\end{proof}

\begin{lemma}\label{lem:2b-A}
 In Case $(A)$, one of the two conditions below holds.
\begin{enumerate}
\renewcommand{\labelenumi}{(\theenumi)}
\item
The $(2,0)$-splitting $H_2$ is weakly $K$-reducible.
\item
$(iii)(E)$, $(iii)(F)$ or $(iii)(G)$ holds.
\end{enumerate}
\end{lemma}

 We have already considered Cases (iii)(E), (iii)(F) and (iii)(G)
in Lemmas \ref{lem:2b-E}, \ref{lem:2b-iiiF} and \ref{lem:2b-iiiG}
respectively.

\begin{proof}
 Performing a $K$-compressing operation on $A_{11}$ 
along a $K$-compressing disk as in the condition (A),
we obtain a disk $D_1$ in $W_1$, 
which is bounded by $l_1$ or $l_2$, say $l_1$,
and intersects $K$ transversely
in at most one point.

 In Cases (i) and (ii),
performing a compressing operation on $A_{12}$,
we obtain a disk $D_2$
which is bounded by $l_2$.
 Hence the disks $D_1$ and $D_2$ show
that $H_2$ is weakly $K$-reducible.

In Case (iv),
performing a $\bdd$-compressing operation on $A_{12}$,
we obtain a disk $D'_2$
whose boundary loop $\bdd D'_2$ is essential in $H_2$
and is disjoint from $l_1$ after an adequate small isotopy.
 Thus the disks $D_1$ and $D'_2$ show
that $H_2$ is weakly $K$-reducible.

 In Case (iii),
the annulus $A_{12}$ is parallel to the annulus $A_2$
in $W_2$.
 By Lemma 2.4 in \cite{GH2},
$H'_2$ is $t_2$-compressible or $t_2$-$\bdd$-compressible in $(V_2, t_2)$.
 In the former case,
let $R$ be a $t_2$-compressing disk of $H'_2$.
If $R$ is contained in $W_2$,
then the disks $D_1$ and $R$ show
that $H_2$ is weakly $K$-reducible.
 Therefore we may assume that $R$ is contained in $W_1$.
 We can isotope $H_1$ in $(M, K)$
along the parallelism between the annuli $A_{12}$ and $A_2$
so that $H_1$ is contained in $W_1$
and so that it is disjoint from $R$.
 Then Proposition \ref{thm:koba} shows
that $H_2$ is weakly $K$-reducible.
 We consider the latter case,
where $H'_2$ has a $t_2$-$\bdd$-compressing disk $R'$
in $(V_2, t_2)$.
If $R'$ is contained in $W_2$,
it is also a $\bdd$-compressing disk of $A_{12}$ in $W_2$
by the definition of $\bdd$-compressing disks.
 We have already considered this case
in the third paragraph in this proof.
 Hence we can assume that $R'$ is contained in $W_1$.
 Then $R'$ is also a $K$-$\bdd$-compressing disk of $A_{11}$ in $(W_1, K)$
by the definition of $\bdd$-compressing disk.
 Thus one of the conditions (E), (F) and (G) holds.
\end{proof}

\begin{lemma}\label{lem:2b-iiB}
In Case $(ii)(B)$,
and hence, also in Case $(ii)(D)$,
$H_2$ is weakly $K$-reducible.
\end{lemma}

\begin{proof}
 In Case (ii),
there is a compressing disk $D_2$ of $A_{12}$
in $W_2 \cap V_2$.
 In Case (B),
there is a $K$-compressing disk $D_1$ of $A_{11}$
in $W_1 \cap V_1$.
 Then these disks $D_1$ and $D_2$ show
that $H_1$ is $K$-reducible.
 Then $H_2$ is weakly $K$-reducible
as shown in the third paragraph 
in the proof of Lemma \ref{lem:2a-A_12compressible}.
 We have this lemma via Lemma \ref{lem:2b-D}.
\end{proof}

\begin{lemma}\label{lem:2b-ivD}
 In Case $(iv)(D)$,
the conclusion $(f)$ of Theorem \ref{thm:l=2} holds.
\end{lemma}

\begin{proof}
 In Case (iv),
there is a $\bdd$-compressing disk $D$ of $A_{12}$ in $W_2$
such that $D$ is contained in $V_2$.
 Let $P=N(\bdd A_{12} \cup (A_{12} \cap \bdd D))$ be a neighborhood
of the union of the two boundary loops $\bdd A_{12}$
and the arc $A_{12} \cup \bdd D$ in $A_{12}$.
 Note that $P$ is the disk with two holes.
 We can isotope $H_1$ in $(M, K)$
along $D$
so that $P$ is isotoped into $H'_2$.
 After this isotopy,
$H_1$ intersects $\text{int}\,W_2$ in an open disk.
 Let $D_2$ be the closure of this open disk.
 Then 
$\bdd D_2$ separates $H_2$ into the once punctured torus
$A_2 \cup (H'_2 \cap P)$ 
and the complementary once punctured torus,
and $D_2$ cuts $W_2$ into two solid tori,
one of which, say $U_1$, contains $A_2$.

 In Case (D),
there is a $K$-$\bdd$-compressing disk $R$ of $A_{11}$
in $(W_1, K)$
such that $R$ is contained in $V_1$.
 The arc $\beta=A_{11} \cap \bdd R$ is essential in $A_{11}$
ignoring the intersection points $K \cap A_{11}$
since the arc $A_2 \cap \bdd R$ is essential in $A_2$
by the definition of $K$-$\bdd$-compressing disk.
 Set $B=N(\beta)$,
the band neighborhood of the arc $\beta$ in $A_{11}$.
 We can isotope $H_1$ in $(M,K)$ along $R$
so that $B$ is isotoped into $A_2$
and that $B \cap P \subset \bdd A_2$.
 After this isotopy,
$H_1$ intersects $\text{int}\,W_1$ in an open disk.
 Let $R_1$ be the closure of this open disk.
 Since $\bdd R_1$ is inessential on $\bdd W_1=H_2$,
$R_1$ is a $\bdd$-parallel disk in $W_1$ ignoring $K$,
and cuts off a $3$-ball $X$ from $W_1$
such that $X$ intersects $K$ in a single arc $t$
which is trivial in $X$.
(See Lemma 3.2 in \cite{GHY}.)

 After these isotopies, 
$H_1$ intersects $H_2$ in the torus with two holes $P \cup B$.
 The solid torus $V_1$ is the union $U_1 \cup X$.
 Hence $t=t_1$.
 We take an arc $\alpha$ in the $3$-ball $X$
so that an endpoint of $\alpha$ in $\text{int}\,t$,
that the other endpoint of $\alpha$ is in the disk $X \cap H_2$
and so that $X$ collapses to $t \cup \alpha$.
 See Figure \ref{fig:Case3}.
 Thus we obtain the conclusion (f) of Theorem \ref{thm:l=2}.
\end{proof}


 We need the next lemma to consider Case (iii)(B).

\begin{lemma}\label{lem:complement}
Suppose that the handlebody $W_1$ contains a separating disk $D$
such that $D$ cuts off from $W_1$
a solid torus $U_1$ disjoint from the knot
and that the complementary $3$-manifold $\text{cl}\,(M-U_1)$
is also a solid torus.
 Then the $(2,0)$-splitting $H_2$ is meridionally stabilized,
and hence $H_2$ is weakly $K$-reducible
$($see Proposition 4.7 in \cite{GH2}\/$)$.
\end{lemma}

\begin{proof}
 The disk $D$ cuts off another solid torus $U_2$ from $W_1$.
 Note that $K$ forms a core of $U_2$ (Lemma 3.3 in \cite{GHY}).
 There is a meridian disk $Q$ of $U_2$
which intersects $K$ transversely in a single point.
 We can take $Q$
so that $\bdd Q$ intersects the disk $D$ in a single arc.
 Let $N(Q)$ be a small regular neighborhood of $Q$ in $U_2$.
 The solid torus $U'_1 = U_1 \cup N(Q)$
intersects $K$ in a trivial arc $s_1$.
 The $3$-ball $\text{cl}(U_2 - N(Q))$ forms
a regular neighborhood of the complementary arc
$s_2 = \text{cl}\,(K - s_1)$
in the complementary solid torus $U'_2 = \text{cl}\,(M - U'_1)$.
 The exterior of $s_2$ in $U'_2$
is homeomorphic to $W_2$.
 Hence we can see that the arc $s_2$ is trivial in $U'_2$,
applying Theorem 1 in \cite{Go}.
 Therefore $(M, K) = (U'_1, s_1) \cup (U'_2, s_2)$ is a $(1,1)$-splitting.
Moreover,
we can take a meridian disk $D_1$ of $s_2$
in $\text{cl}\,(U_2 - N(Q))$
and a canceling disk $D_2$ of $s_2$
in $\text{cl}\,(U'_2 - \text{cl}\,(U_2 - N(Q))) = W_2$
so that $\bdd D_1$ and $\bdd D_2$ intersects
transversely in a single point.
 These disks $D_1$ and $D_2$ show
that $H_2$ is meridionally stabilized.
\end{proof}

\begin{lemma}\label{lem:2b-iB}
 In Cases $(i)(B)$ and $(iii)(B)$, 
one of the following four conditions holds.
\begin{enumerate}
\renewcommand{\labelenumi}{(\theenumi)}
\item
 The $(2,0)$-splitting $H_2$ is weakly $K$-reducible.
\item
 The $(1,1)$-splitting $H_1$ has a satellite diagram.
\item
 One of the conditions $(E)$, $(F)$ and $(G)$ holds.
\item
 The conditions $(iv)$ and $(D)$ hold.
\end{enumerate}
\end{lemma}

 We have already considered the cases 
(i)(E), (i)(F), (i)(G), 
(iii)(E), (iii)(F), (iii)(G) 
and (iv)(D)
in Lemmas \ref{lem:2b-E}, \ref{lem:2b-iFiiFiGiiG}, 
\ref{lem:2b-iiiF}, \ref{lem:2b-iiiG}
and \ref{lem:2b-ivD}.
 In the case of the conclusion (2), 
we have the conclusion (b), (c), (d) or (e) of Theorem \ref{thm:l=2}
by Proposition 4.9 in \cite{GH2}.

\begin{proof}
 In Case (i),
there is a compressing disk $D$ of $A_{12}$ in $V_1$.
 Compressing a copy of $A_{12}$ along $D$,
we obtain disks $D_1$ and $D_2$
bounded by the intersection loops $H_1 \cap H_2 = l_1 \cup l_2$.
 Note that these disks $D_1$ and $D_2$ are contained in $W_2 \cap V_1$,
and form compressing disks of $A_2$.

 In Case (iii),
the annulus $A_{12}$ is parallel to the annulus $A_2$ in $W_2$.

 In Case (B),
there is a $K$-compressing disk $R$ of $A_{11}$ in $V_1$
such that $\bdd R$ bounds in $A_{11}$ a disk $R'$
which intersects $K$ in precisely two points.
 Compressing a copy of $A_{11}$ along $R$,
we obtain an annulus $A$ which is disjoint from $K$.

 The surface $H'_2$ is 
$t_2$-compressible or $t_2$-$\bdd$-compressible in $(V_2, t_2)$
by Lemma 2.4 in \cite{GH2}.
 Suppose first
that $H'_2$ has a $t_2$-compressing or $t_2$-$\bdd$-compressing
disk $P$ in $W_1$.
 When $P$ is a $t_2$-compressing disk,
$D_1$ and $P$ show that $H_2$ is weakly $K$-reducible in Case (i),
and in Case (iii) we isotope $H_1$
along the parallelism between $A_{12}$ and $A_2$
so that $H_1$ is contained in int\,$W_1$
and that $H_1$ is disjoint from $P$, 
to see that $H_2$ is weakly $K$-reducible
by Proposition \ref{thm:koba}.
 This is the conclusion (1) of this lemma.
 When $P$ is a $t_2$-$\bdd$-compressing disk,
it is also a $K$-$\bdd$-compressing disk of $A_{11}$
by the definition of $t_2$-$\bdd$-compressing disk.
 Thus one of the conditions (E), (F) and (G) holds.
 This is the conclusion (3) of this lemma.

\begin{figure}[htbp]
\centering
\includegraphics[width=.5\textwidth]{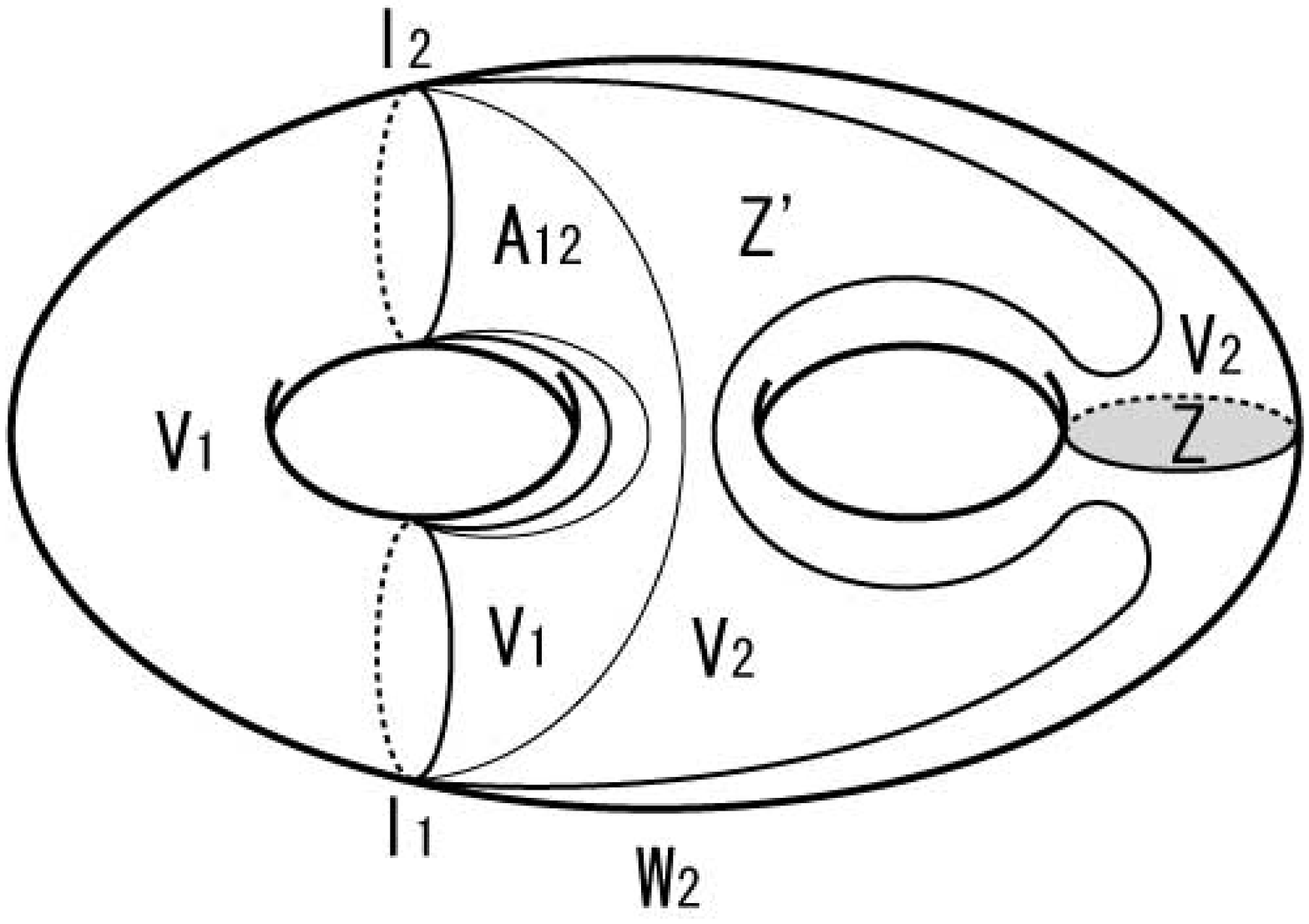}
\caption{}
\label{fig:13-11}
\end{figure}

 Hence we may assume
that $H'_2$ has a $t_2$-compressing or $t_2$-$\bdd$-compressing
disk $Z$ in $W_2$.
 When $Z$ is a $t_2$-$\bdd$-compressing disk of $H'_2$,
it is also a $\bdd$-compressing disk of $A_{12}$.
 Performing the $\bdd$-compressing operation on a copy of $A_{12}$ along $Z$,
we obtain a compressing disk of $H'_2$ in $W_2 \cap V_2$.
 Thus we may assume that $Z$ is a $t_2$-compressing disk of $H'_2$.
 If $\bdd Z$ is parallel to a component of $\bdd H'_2$,
then $Z\,(\subset V_2)$ and 
the $K$-compressing disk $R\,(\subset V_1)$ of $A_{11}$
show that $H_1$ is $K$-reducible.
 Hence $H_2$ is weakly $K$-reducible
as shown in the third paragraph 
in the proof of Lemma \ref{lem:2a-A_12compressible}.
 So we may assume
that $\bdd Z$ is not parallel to a component of $\bdd H'_2$.
 Compressing $H'_2$ along $R_2$,
we obtain an annulus $Z'$ in $W_2 \cap V_2$
such that $\bdd Z' = \bdd H'_2$.
 We isotope $\text{int}\,Z'$ slightly into $\text{int}\,(V_2 \cap W_2)$
so that $Z' \cap H'_2 = \bdd Z' = \bdd H'_2$.
 A typical example is described in Figure \ref{fig:13-11}.
 This annulus $Z'$ is 
$t_2$-compressible or $t_2$-$\bdd$-compressible in $(V_2, t_2)$
by Lemma 2.4 in \cite{GH2}.
 In the former case,
performing a $t_2$-compressing on $Z'$,
we obtain a disk
which is disjoint from $K$ and bounded by a loop of $\bdd H'_2$.
 This disk and $R$ 
show that $H_1$ is $K$-reducible.
 This implies the conclusion (1) again.
 In the latter case,
let $Q$ be a $t_2$-$\bdd$-compressing disk of $Z'$.
 If the arc $\bdd Q \cap H_1$ is contained in $A_{11}$,
then performing the $t_2$-$\bdd$-compressing operation on $Z'$ along $Q$,
we obtain a disk $Q'\,(\subset V_2)$
such that $\bdd Q'$ bounds a disk in $A_{11}$,
which intersects $K$ in two points.
 Note that $Q'$ is disjoint from $K$.
 The disks $R$ and $Q'$ show
that $H_1$ has a satellite diagram on $A_{11}$.
 (In fact, for $i=1$ and $2$,
we can take a canceling disk $C_i$ of $t_i$ in $(V_i, t_i)$
so that $C_1$ is disjoint from $R$
and $C_2$ is disjoint from $Q$.)
 This is the conclusion (2) of this lemma.
 Then we may assume
that the arc $\bdd Q \cap H_1$ is contained in $A_{12}$.
 This implies that the annulus $Z'$ is parallel to $A_{12}$ in $W_2$.
%
 Since $Z'$ is obtained by a compression on $H'_2$,
it has a $\bdd$-compressing disk $G$ in $(V_2, t_2)$
such that it is contained in $W_2 \cap V_2$ and $\bdd G \cap H_2$
is an essential arc in $H'_2$.
 (There is an arc connecting the two loops $l_1$ and $l_2$ in $H'_2$
such that it is disjoint from $\bdd Z$.)
 Therefore $A_{12}$ also has a $\bdd$-compressing disk in $W_2 \cap V_2$.
 Thus the condition (iv) holds.

 The annulus $A$,
which was obtained from $A_{11}$ by $K$-compressing along $R$,
is $K$-compressible or $K$-$\bdd$-compressible in $(W_1, K)$
by Lemma 3.1 in \cite{GH2}.
 In the former case,
performing the $K$-compressing operation on $A$,
we obtain a $K$-compressing disk bounded by $l_1$.
 Then this disk and $Z$ show
that $H_2$ is weakly $K$-reducible.
 This is the conclusion (1).
 In the latter case,
let $C$ be a $K$-$\bdd$-compressing disk of $A$.
 First suppose 
that the arc $\bdd C \cap H_2$ is contained in $A_2$.
 We can take $C$ to be disjoint from the copy of $R$ in $A$.
 Then $C$ forms
a $K$-$\bdd$-compressing disk of the annulus $A_{11}$
such that the arc $A_{11} \cap \bdd C$ is essential in $A_{11}$.
 Thus the condition (D) holds,
and we obtain the conclusion (4) of this lemma.
 Hence we may assume
that the arc $\bdd C \cap H_2$ is contained in $H'_2$.
 In Case (i),
$\bdd$-compressing $A$ along $C$,
we obtain an essential disk disjoint from $K$ in $W_1$.
 An adequate small isotopy moves this disk so that it is disjoint from $l_1$.
 Hence, this disk together with $D_1$
shows that $H_2$ is weakly $K$-reducible.
 We consider Case (iii).
 In this case, we will move $H_1$ ignoring $K$
so that we can use Lemma \ref{lem:complement}.
 Recall that $R$ is a $K$-compressing disk of $A_{11}$
such that $\bdd R$ bounds a disk $R'$ on $A_{11}$.
 In $W_1$
the $2$-sphere $R \cup R'$ bounds a $3$-ball,
and hence $A_{11}$ and $A$ are isotopic in $W_1$
fixing their boundary loops $\bdd A_{11} = \bdd A$
ignoring $K$.
 Hence $H_1$ is isotopic to the torus $H=A \cup A_{12}$ in $M$, 
ignoring $K$.
 Since we are in Case (iii),
we can isotope $H$
along the parallelism between the annuli $A_{12}$ and $A_2$
so that $A_{12}$ is isotoped onto $A_2$.
 Recall that the $K$-$\bdd$-compressing disk $C$ of $A$
intersects $H_2$ in an essential arc in $H'_2$.
 (See Figure \ref{fig:13-12} for a typical example.)
 We isotope $H$ along $C$
so that $H \cap H_2$ is a torus with one hole $H_0$
and that $H \cap \text{int}\,W_1$ is an open disk
the closure of which is an essential separating disk $C'$
bounded by the loop $\bdd H_0$.
 Then the solid torus $V_1$ is isotopic to the solid torus $U_1$
bounded by the torus $H = H_0 \cup C'$ in $W_1$,
ignoring $K$.
 The complementary $3$-manifold $\text{cl}\,(M-U_1)$
is isotopic to the solid torus $V_2$.
 Then Lemma \ref{lem:complement} shows
that $H_2$ is meridionally stabilized,
and hence $H_2$ is weakly $K$-reducible
by Proposition 4.7 in \cite{GH2}.
\end{proof}

\begin{figure}[htbp]
\centering
\includegraphics[width=.5\textwidth]{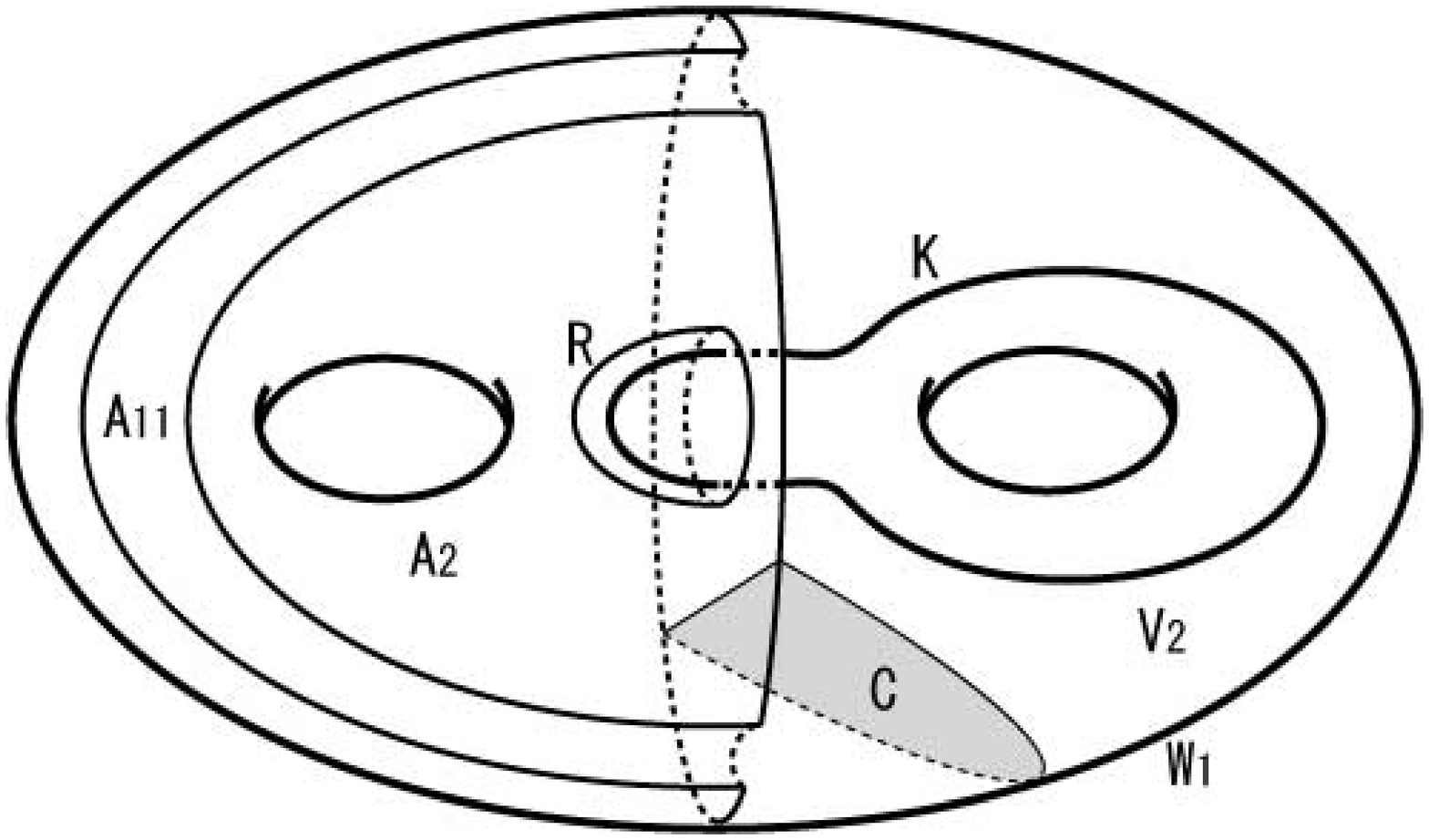}
\caption{}
\label{fig:13-12}
\end{figure}

\begin{remark}
 In Case (i)(B), 
we can delete the conclusion (2) of the above lemma.
 In fact, we can see the annuli $Z'$ and $A_{12}$ are parallel in $W_2$
as below.
 The torus $A_{12} \cup Z'$ bounds in $W_2$
a $3$-manifold that is homeomorphic to
an exterior $E$ of a (possibly trivial)
knot in $S^3$,
because a handlebody is irreducible.
 Since the loop $l_1$ bounds the disk $D_1$ in $W_2$,
it is of meridional slope
on the boundary torus of the knot exterior $E$.
 Because $l_1$ is a meridian of the solid torus $V_1$
and $M$ is not homeomorphic to $S^2 \times S^1$,
$l_1$ is not a meridian of the solid torus $V_2$.
 Since $E$ is cut off from $V_2$ by $Z'$,
it is not homeomorphic to the exterior of a non-trivial knot exterior,
but is the exterior of the trivial knot in $S^3$
with $l_1$ being a meridian of the knot.
 Hence $E$ is a solid torus with $l_1$ being a longitude,
and hence $A_{12}$ and $Z'$ are parallel in $W_2$.
\end{remark}


\bibliographystyle{amsplain}




\end{document}